\documentclass{amsart}
\usepackage{amsmath,amssymb,amsthm,bm,bbm}
\usepackage{amscd}
\usepackage{graphicx,enumerate}
\usepackage{multirow}
\usepackage{layout}
\usepackage[all]{xy}
\usepackage[OT2,T1]{fontenc}
\DeclareSymbolFont{cyrletters}{OT2}{wncyr}{m}{n}
\DeclareMathSymbol{\Sha}{\mathalpha}{cyrletters}{"58}
\usepackage[OT2,OT1]{fontenc}
\newcommand\cyr{\renewcommand\rmdefault{wncyr}
\renewcommand\sfdefault{wncyss}
\renewcommand\encodingdefault{OT2}
\normalfont\selectfont}
\DeclareTextFontCommand{\textcyr}{\cyr}

\usepackage{hyperref}
\usepackage{url}

\theoremstyle{plain}
\newtheorem{theorem}{Theorem}[section]
\newtheorem*{theorem-nn}{Theorem}
\newtheorem{lemma}[theorem]{Lemma}
\newtheorem{proposition}[theorem]{Proposition}
\newtheorem*{proposition-nn}{Proposition}
\newtheorem{corollary}[theorem]{Corollary}

\theoremstyle{definition}

\newtheorem{example}[theorem]{Example}
\newtheorem{remark}[theorem]{Remark}

\theoremstyle{remark}

\newcommand{\bZ}{\mathbbm{Z}}\newcommand{\bQ}{\mathbbm{Q}}
\newcommand{\bC}{\mathbbm{C}}
\newcommand{\bG}{\mathbbm{G}}\newcommand{\bF}{\mathbbm{F}}
\newcommand{\bA}{\mathbbm{A}}

\newcommand{\oa}{\overline{a}}
\newcommand{\ob}{\overline{b}}
\newcommand{\oc}{\overline{c}}
\newcommand{\oz}{\overline{z}}

\title{Hasse norm principle for Heisenberg extensions of degree $p^3$}

\author[A. Hoshi]{Akinari Hoshi}
\address{Department of Mathematics, Niigata University, Niigata 950-2181, Japan}
\email{hoshi@math.sc.niigata-u.ac.jp}

\author[A. Yamasaki]{Aiichi Yamasaki}
\address{Department of Mathematics, Kyoto University, Kyoto 606-8502, Japan}
\email{aiichi.yamasaki@gmail.com}

\thanks{
This work was partially supported by JSPS KAKENHI Grant Numbers 
19K03418, 20H00115, 20K03511, 24K00519, 24K06647.
}


\begin{document}
\maketitle
\begin{abstract}
Let $k$ be a global field and $p$ be an odd prime number. 
We give a necessary and sufficient condition 
for the Hasse norm principle for separable field extensions $K/k$, 
i.e. the determination of the Shafarevich-Tate group $\Sha(T)$ of 
the norm one tori $T=R^{(1)}_{K/k}(\bG_{m,K})$ of $K/k$, 
with $[K:k]=p^3$ or $p^2$ 
when the Galois group 
of the Galois closure 
of $K/k$ is the Heisenberg group 
$E_p(p^3)\simeq (C_p)^2\rtimes C_p$ of order $p^3$, 
i.e. the extraspecial group of order $p^3$ with exponent $p$. 
As a consequence, we get the 
Tamagawa number $\tau(T)=p^2$, $p$ or $1$ 
via Ono's formula $\tau(T)=|H^1(k,\widehat{T})|/|\Sha(T)|$. 
\end{abstract}
%
%
\section{Introduction}\label{S1}

Let $k$ be a global field, 
i.e. a number field (a finite extension of $\bQ$) 
or a function field of an algebraic curve over 
$\bF_q$ (a finite extension of $\bF_q(t))$.
Let $K/k$ be a finite separable field extension and 
$\bA_K^\times$ be the idele group of $K$. 
We say that {\it the Hasse norm principle holds for $K/k$} 
if $(N_{K/k}(\bA_K^\times)\cap k^\times)/N_{K/k}(K^\times)=1$ 
where $N_{K/k}$ is the norm map. 
%
Hasse \cite[Satz, page 64]{Has31} proved that 
the Hasse norm principle holds for any cyclic extension $K/k$ 
but does not hold for bicyclic extension $\bQ(\sqrt{-39},\sqrt{-3})/\bQ$. 

Let $\overline{k}$ be a fixed separable closure of $k$. 
Let $T$ be an algebraic $k$-torus, 
i.e. a group $k$-scheme with fiber product (base change) 
$T\times_k \overline{k}=
T\times_{{\rm Spec}\, k}\,{\rm Spec}\, \overline{k}
\simeq (\bG_{m,\overline{k}})^n$; 
$k$-form of the split torus $(\bG_{m,k})^n$. 
%
Let $E$ be a principal homogeneous space (= torsor) under $T$.  
{\it Hasse principle holds for $E$} means that 
if $E$ has a $k_v$-rational point for all completions $k_v$ of $k$, 
then $E$ has a $k$-rational point. 
The set $H^1(k,T)$ classifies all such torsors $E$ up 
to (non-unique) isomorphism. 
We take {\it the Shafarevich-Tate group} of $T$: 
\begin{align*}
\Sha(T)={\rm Ker}\left\{H^1(k,T)\xrightarrow{\rm res} \bigoplus_{v\in V_k} 
H^1(k_v,T)\right\}
\end{align*}
where $V_k$ is the set of all places of $k$ and 
$k_v$ is the completion of $k$ at $v$. 
Then 
Hasse principle holds for all torsors $E$ under $T$ 
if and only if $\Sha(T)=0$. 

Let $T=R^{(1)}_{K/k}(\bG_{m,K})$ be the norm one torus of $K/k$,
i.e. the kernel of the norm map $R_{K/k}(\bG_{m,K})\rightarrow \bG_{m,k}$ 
where 
$R_{K/k}$ is the Weil restriction 
(see Ono \cite[Section 1.4]{Ono61}, Voskresenskii \cite[page 37, Section 3.12]{Vos98}). 
It is biregularly isomorphic to the norm hypersurface 
$f(x_1,\ldots,x_n)=1$ where 
$f\in k[x_1,\ldots,x_n]$ is the polynomial of total 
degree $n$ defined by the norm map $N_{K/k}:K^\times\to k^\times$. 
The norm one torus $T=R^{(1)}_{K/k}(\bG_{m,K})$ has the 
Chevalley module $J_{G/H}$ as its character module 
and the field $L(J_{G/H})^G\simeq k(T)$ as its function field 
where 
$L$ is the splitting field of $T$ with 
$G={\rm Gal}(L/k)$ and $H={\rm Gal}(L/K)\lneq G$ 
and 
$J_{G/H}=(I_{G/H})^\circ={\rm Hom}_\bZ(I_{G/H},\bZ)$ 
is the dual lattice of $I_{G/H}={\rm Ker}\ \varepsilon$ and 
$\varepsilon : \bZ[G/H]\rightarrow \bZ$, 
$\sum_{\overline{g}\in G/H} a_{\overline{g}}\,\overline{g}\mapsto \sum_{\overline{g}\in G/H} a_{\overline{g}}$ 
is the augmentation map. 
We have the exact sequence $0\rightarrow \bZ\rightarrow \bZ[G/H]
\rightarrow J_{G/H}\rightarrow 0$ and ${\rm rank}_{\bZ}\, J_{G/H}=n-1$. 
More precisely, write $J_{G/H}=\oplus_{1\leq i\leq n-1}\bZ u_i$ and 
define the action of $G$ on $L(x_1,\ldots,x_{n-1})$ by 
$x_i^{\sigma}=\prod_{j=1}^{n-1} x_j^{a_{i,j}} (1\leq i\leq n-1)$ 
for any $\sigma\in G$, 
when $u_i^{\sigma}=\sum_{j=1}^{n-1} a_{i,j} u_j$ $(a_{i,j}\in\bZ)$. 
Then the invariant field $L(x_1,\ldots,x_{n-1})^G\simeq k(T)$: 
the function field of $T=R^{(1)}_{K/k}(\bG_{m,K})$ 
(see Endo and Miyata \cite[Section 1]{EM73}, \cite[Section 1]{EM75} 
and Voskresenskii \cite[Section 4.8]{Vos98}). 

Ono \cite{Ono63} established the relationship 
between the Hasse norm principle for $K/k$ 
and Hasse principle for all torsors $E$ under
the norm one torus $R^{(1)}_{K/k}(\bG_{m,K})$: 
\begin{theorem}[{Ono \cite[page 70]{Ono63}, see also Platonov \cite[page 44]{Pla82}, Kunyavskii \cite[Remark 3]{Kun84}, Platonov and Rapinchuk \cite[page 307]{PR94}}]\label{thOno}
Let $k$ be a global field and $K/k$ be a finite separable extension. 
Let $T=R^{(1)}_{K/k}(\bG_{m,K})$ be the norm one torus of $K/k$. 
Then 
\begin{align*}
\Sha(T)\simeq (N_{K/k}(\bA_K^\times)\cap k^\times)/N_{K/k}(K^\times).
\end{align*}
In particular, $\Sha(T)=0$ if and only if 
the Hasse norm principle holds for $K/k$. 
\end{theorem}

For Galois extensions $K/k$, Tate \cite{Tat67} gave the following theorem:
%

\begin{theorem}[{Tate \cite[page 198]{Tat67}}]\label{thTate}
Let $k$ be a global field, $K/k$ be a finite Galois extension 
with Galois group $G={\rm Gal}(K/k)$. 
Let $V_k$ be the set of all places of $k$ 
and $G_v$ be the decomposition group of $G$ at $v\in V_k$. 
Then 
\begin{align*}
(N_{K/k}(\bA_K^\times)\cap k^\times)/N_{K/k}(K^\times)\simeq 
{\rm Coker}\left\{\bigoplus_{v\in V_k}\widehat H^{-3}(G_v,\bZ)\xrightarrow{\rm cores}\widehat H^{-3}(G,\bZ)\right\}
\end{align*}
where $\widehat H$ is the Tate cohomology. 
In particular, the Hasse norm principle holds for $K/k$ 
if and only if the restriction map 
$H^3(G,\bZ)\xrightarrow{\rm res}\bigoplus_{v\in V_k}H^3(G_v,\bZ)$ 
is injective. 
In particular, if 
$H^3(G,\bZ)=0$, then 
the Hasse norm principle holds for $K/k$. 
\end{theorem}
Let $C_n$ be the cyclic group of order $n$. 
If $G\simeq C_n$, then 
$\widehat H^{-3}(G,\bZ)\simeq H^3(G,\bZ)\simeq H^1(G,\bZ)=0$ 
and hence Hasse's original theorem follows. 
If there exists a place $v$ of $k$ such that $G_v=G$, then 
the Hasse norm principle also holds for $K/k$. 
For example, the Hasse norm principle holds for $K/k$ with 
$G\simeq V_4=(C_2)^2$ if and only if 
there exists a $($ramified$)$ place $v$ of $k$ such that $G_v=V_4$ because 
$H^3(V_4,\bZ)\simeq\bZ/2\bZ$ and $H^3(C_2,\bZ)=0$. 
The Hasse norm principle holds for $K/k$ with 
$G\simeq (C_2)^3$ if and only if {\rm (i)} 
there exists a $($ramified$)$ place $v$ of $k$ such that $G_v=G$ 
or {\rm (ii)} 
there exist $($ramified$)$ places $v_1,v_2,v_3$ of $k$ such that 
$G_{v_i}\simeq V_4$ and 
$H^3(G,\bZ)\xrightarrow{\rm res}
H^3(G_{v_1},\bZ)\oplus H^3(G_{v_2},\bZ)\oplus H^3(G_{v_3},\bZ)$ 
is an isomorphism because 
$H^3(G,\bZ)\simeq(\bZ/2\bZ)^{\oplus 3}$ and 
$H^3(V_4,\bZ)\simeq \bZ/2\bZ$. 
Note that if $v$ is an unramified place of $k$, then $G_v$ is cyclic. 

Let $T$ be an algebraic $k$-torus 
and $T(k)$ be the group of $k$-rational points of $T$. 
Then $T(k)$ 
embeds into $\prod_{v\in V_k} T(k_v)$ by the diagonal map 
where 
$V_k$ is the set of all places of $k$ and 
$k_v$ is the completion of $k$ at $v$. 
Let $\overline{T(k)}$ be the closure of $T(k)$  
in the product $\prod_{v\in V_k} T(k_v)$. 
The group 
\begin{align*}
A(T)=\left(\prod_{v\in V_k} T(k_v)\right)/\overline{T(k)}
\end{align*}
is called {\it the kernel of the weak approximation} of $T$. 
We say that {\it $T$ has the weak approximation property} if $A(T)=0$. 

%
\begin{theorem}[{Voskresenskii \cite[Theorem 5, page 1213]{Vos69}, 
\cite[Theorem 6, page 9]{Vos70}, see also \cite[Section 11.6, Theorem, page 120]{Vos98}}]\label{thV}
Let $k$ be a global field, 
$T$ be an algebraic $k$-torus and $X$ be a smooth $k$-compactification of $T$. 
Then there exists an exact sequence
\begin{align*}
0\to A(T)\to H^1(k,{\rm Pic}\,\overline{X})^{\vee}\to \Sha(T)\to 0
\end{align*}
where $M^{\vee}={\rm Hom}(M,\bQ/\bZ)$ is the Pontryagin dual of $M$. 
In particular, if $T$ is retract $k$-rational, then $ H^1(k,{\rm Pic}\,\overline{X})=0$ and hence 
$A(T)=0$ and $\Sha(T)=0$. 
Moreover, if $L$ is the splitting field of $T$ and $L/k$ 
is an unramified extension, then $A(T)=0$ and 
$H^1(k,{\rm Pic}\,\overline{X})^{\vee}\simeq \Sha(T)$. 
\end{theorem}
For the last assertion, see Voskresenskii \cite[Section 11.5]{Vos98}. 
It follows that 
$H^1(k,{\rm Pic}\,\overline{X})=0$ if and only if $A(T)=0$ and $\Sha(T)=0$, 
i.e. $T$ has the weak approximation property and 
Hasse principle holds for all torsors $E$ under $T$. 
If $T$ is (stably/retract) $k$-rational, 
then $H^1(k,{\rm Pic}\,\overline{X})=0$ 
(see Voskresenskii \cite[Theorem 5, page 1213]{Vos69},
Manin \cite[Section 30]{Man74}, 
Manin and Tsfasman \cite{MT86} 
and also Hoshi, Kanai and Yamasaki \cite[Section 1]{HKY22}). 
Theorem \ref{thV} was generalized 
to the case of linear algebraic groups by Sansuc \cite{San81}.



Applying Theorem \ref{thV} to $T=R^{(1)}_{K/k}(\bG_{m,K})$,  
it follows from Theorem \ref{thOno} that 
$H^1(k,{\rm Pic}\,\overline{X})=0$ if and only if 
$A(T)=0$ and $\Sha(T)=0$, 
i.e. 
$T$ has the weak approximation property and 
the Hasse norm principle holds for $K/k$. 
In the algebraic language, 
the latter condition $\Sha(T)=0$ means that 
for the corresponding norm hypersurface $f(x_1,\ldots,x_n)=b$, 
it has a $k$-rational point 
if and only if it has a $k_v$-rational point 
for any place $v$ of $k$ where 
$f\in k[x_1,\ldots,x_n]$ is the polynomial of total 
degree $n$ defined by the norm map $N_{K/k}:K^\times\to k^\times$ 
and $b\in k^\times$ 
(see Voskresenskii \cite[Example 4, page 122]{Vos98}).


When $K/k$ is a finite Galois extension, 
we have: 

\begin{theorem}[{Voskresenskii \cite[Theorem 7]{Vos70}, Colliot-Th\'{e}l\`{e}ne and Sansuc \cite[Proposition 1]{CTS77}}]\label{thV2}
Let $k$ be a field and 
$K/k$ be a finite Galois extension with Galois group $G={\rm Gal}(K/k)$. 
Let $T=R^{(1)}_{K/k}(\bG_{m,K})$ be the norm one torus of $K/k$ 
and $X$ be a smooth $k$-compactification of $T$. 
Then 
$H^1(H,{\rm Pic}\, X_K)\simeq H^2(H,\widehat{T})\simeq H^3(H,\bZ)$ 
for any subgroup $H$ of $G$. 
In particular, 
$H^1(k,{\rm Pic}\, \overline{X})\simeq
H^1(G,{\rm Pic}\, X_K)\simeq H^2(G,\widehat{T})
\simeq H^2(G,J_G)\simeq H^3(G,\bZ)\simeq M(G)$ 
where $M(G)$ is the Schur multiplier of $G$.
\end{theorem}
In other words, for the $G$-lattice $J_G\simeq \widehat{T}$, 
$H^1(H,[J_G]^{fl})\simeq H^3(H,\bZ)$ for any subgroup $H$ of $G$ 
and $H^1(G,[J_G]^{fl})\simeq H^3(G,\bZ)\simeq H^2(G,\bQ/\bZ)$; 
the Schur multiplier of $G$. 
By the exact sequence $0\to\bZ\to\bZ[G]\to J_G\to 0$, 
we also have $\delta:H^1(G,J_G)\simeq H^2(G,\bZ)\simeq G^{ab}\simeq 
G/[G,G]$ where $\delta$ is the connecting homomorphism and 
$G^{ab}$ is the abelianization of $G$. 


By Poitou-Tate duality (see Milne \cite[Theorem 4.20]{Mil86}, 
Platonov and Rapinchuk \cite[Theorem 6.10]{PR94}, 
Neukirch, Schmidt and Wingberg \cite[Theorem 8.6.8]{NSW00}, 
Harari \cite[Theorem 17.13]{Har20}), 
we also have 
\begin{align*}
\Sha(T)^\vee\simeq\Sha^2(G,\widehat{T})
\end{align*}
where $\Sha(T)^\vee={\rm Hom}(\Sha(T),\bQ/\bZ)$ and 
\begin{align*}
\Sha^i(G,\widehat{T})={\rm Ker}\left\{H^i(G,\widehat{T})\xrightarrow{\rm res} \bigoplus_{v\in V_k} 
H^i(G_v,\widehat{T})\right\}\quad (i\geq 1)
\end{align*}
is {\it the $i$-th Shafarevich-Tate group} 
of $\widehat{T}={\rm Hom}(T,\bG_{m,K})$, 
$G={\rm Gal}(K/k)$ and $K$ is the minimal splitting field of $T$. 
Note that $\Sha(T)^\vee=\Sha^1(G,T)^\vee\simeq \Sha^2(G,\widehat{T})$. 
In the special case where 
$T=R^{(1)}_{K/k}(\bG_{m,K})$ and $K/k$ is Galois with $G={\rm Gal}(K/k)$, 
we have 
$H^2(G,\widehat{T})=H^2(G,J_{G})\simeq H^3(G,\bZ)$ and hence 
we get Tate's theorem (Theorem \ref{thTate}) 
via Ono's theorem (Theorem \ref{thOno}). 


Let 
$M$ be a $G$-lattice, i.e. 
finitely generated $\bZ[G]$-module which is $\bZ$-free as an abelian group. 
We define 
\begin{align*}
\Sha^i_\omega(G,M):={\rm Ker}\left\{H^i(G,M)\xrightarrow{{\rm res}}\bigoplus_{H\leq G:{\rm\, cyclic}}
H^i(H,M)\right\}\quad (i\geq 1) .
\end{align*}
Note that ``$\Sha^i_\omega$'' corresponds to 
the unramified part of ``$\Sha^i$'' because 
if $v\in V_k$ is unramified, then $G_v\simeq C_n$ and 
all the cyclic subgroups of $G$ appear as $G_v$ 
from the Chebotarev density theorem. 
%
\begin{theorem}[{Colliot-Th\'{e}l\`{e}ne and Sansuc \cite[Proposition 9.5 (ii)]{CTS87}, see also \cite[Proposition 9.8]{San81}, \cite[page 98]{Vos98}, \cite[Corollaire 1]{CTHS05}, \cite[Theorem 2.3]{BP20}}]\label{thCTS87}
Let $k$ be a field 
and $K/k$ be a finite Galois extension 
with Galois group $G={\rm Gal}(K/k)$. 
Let $T$ be an algebraic $k$-torus which splits over $K$ and 
$X$ be a smooth $k$-compactification of $T$. 
Then we have 
\begin{align*}
\Sha^2_\omega(G,\widehat{T})\simeq 
H^1(G,{\rm Pic}\, X_K)\simeq {\rm Br}(X)/{\rm Br}(k)
\end{align*}
where 
${\rm Br}(X)$ is the \'etale cohomological Brauer group of $X$ 
$($it is the same as the Azumaya-Brauer group of $X$ 
for such $X$, see \cite[page 199]{CTS87}$)$. 
\end{theorem}

In other words, 
we have 
$H^1(k,{\rm Pic}\, \overline{X})\simeq H^1(G,{\rm Pic}\, X_K)\simeq 
H^1(G,[\widehat{T}]^{fl})\simeq \Sha^2_\omega(G,\widehat{T})\simeq {\rm Br}(X)/{\rm Br}(k)$. 
We also see  
${\rm Br}_{\rm nr}(k(X)/k)={\rm Br}(X)\subset {\rm Br}(k(X))$ 
(see Saltman \cite[Proposition 10.5]{Sal99}, 
Colliot-Th\'{e}l\`{e}ne \cite[Theorem 5.11]{CTS07}, 
Colliot-Th\'{e}l\`{e}ne and Skorobogatov 
\cite[Proposition 6.2.7]{CTS21}).
Moreover, by taking the duals of Voskresenskii's exact sequence as in Theorem \ref{thV}, 
we get the following exact sequence
\begin{align*}
0\to \Sha(T)^\vee\simeq \Sha^2(G,\widehat{T})\to \Sha^2_\omega(G,\widehat{T})\to A(T)^\vee\to 0
\end{align*}
where $M^{\vee}={\rm Hom}(M,\bQ/\bZ)$ 
and 
the map $\Sha^2(G,\widehat{T})\to \Sha^2_\omega(G,\widehat{T})$ 
is the natural inclusion arising from the Chebotarev density theorem 
(see also Macedo and Newton \cite[Proposition 2.4]{MN22}). 

For norm one tori $T=R^{(1)}_{K/k}(\bG_{m,K})$, 
we also obtain the group $T(k)/R$ of $R$-equivalence classes 
over a local field $k$ via 
$T(k)/R\simeq H^1(k,{\rm Pic}\,\overline{X})\simeq 
H^1(G,[\widehat{T}]^{fl})$ 
(see Colliot-Th\'{e}l\`{e}ne and Sansuc \cite[Corollary 5, page 201]{CTS77}, 
Voskresenskii \cite[Section 17.2]{Vos98} and Hoshi, Kanai and Yamasaki \cite[Section 7, Application 1]{HKY22}). 

We return to the Hasse norm principle for $K/k$. 

The Hasse norm principle for Galois extensions $K/k$ 
was investigated by Gerth \cite{Ger77}, \cite{Ger78} and 
Gurak \cite{Gur78a}, \cite{Gur78b}, \cite{Gur80} 
(see also \cite[pages 308--309]{PR94}). 
Gurak \cite[Corollary 2.2]{Gur78a} showed that 
the Hasse norm principle holds for a Galois extension $K/k$ 
if the restriction map 
$H^3(G_p,\bZ)\xrightarrow{\rm res}\bigoplus_{v\in V_k}H^3(G_p$ $\cap$ $G_v,\bZ)$ 
is injective for any $p\mid |G|$ where $G_p$ is a $p$-Sylow subgroup of $G={\rm Gal}(K/k)$. 
In particular, because $H^3(C_n,\bZ)=0$, 
the Hasse norm principle holds for a Galois extension $K/k$ 
if all the Sylow subgroups of $G={\rm Gal}(K/k)$ are cyclic. 
Recall that, by Tate's theorem (Theorem \ref{thTate}), 
the Hasse norm principle holds for $K/k$ 
if and only if the restriction map 
$H^3(G,\bZ)\xrightarrow{\rm res}\bigoplus_{v\in V_k}H^3(G_v,\bZ)$ 
is injective. 
This is also equivalent to $\Sha(T)=0$ 
by Ono's theorem (Theorem \ref{thOno}) 
via $\widehat{T}\simeq J_{G}$ and 
$\Sha(T)^\vee\simeq \Sha^1(G,T)^\vee\simeq\Sha^2(G,\widehat{T})$ 
and $H^2(G,\widehat{T})=H^2(G,J_{G})\simeq H^3(G,\bZ)$ 
where $T=R^{(1)}_{K/k}(\bG_{m,K})$ is the norm one torus of $K/k$ 
and 
$M^{\vee}={\rm Hom}(M,\bQ/\bZ)$. 
%
Note that 
$\Sha^2_\omega(G,\widehat{T})=0$ (and hence $A(T)=0$ and $\Sha(T)=0$) 
also follows 
from the retract $k$-rationality of $T=R^{(1)}_{K/k}(\bG_{m,K})$ 
which is equivalent to that all the Sylow subgroups of $G$ is cyclic 
due to Endo and Miyata \cite[Theorem 2.3]{EM75}. 
For the rationality problem for norm one tori $T=R^{(1)}_{K/k}(\bG_{m,K})$, 
see e.g. 
\cite{EM75}, \cite{CTS77}, \cite{Hur84}, \cite{CTS87}, 
\cite{LeB95}, \cite{CK00}, \cite{LL00}, \cite{Flo}, \cite{End11}, 
\cite{HY17}, \cite{HHY20}, \cite{HY21}, \cite{HY24}, \cite{HY1}, \cite{HY2}. 

For non-Galois extension $K/k$ with $[K:k]=n$, 
the Hasse norm principle was investigated by 
Bartels \cite[Lemma 4]{Bar81a} (holds for $n=p$; prime), 
\cite[Satz 1]{Bar81b} (holds for $G\simeq D_n$), 
Voskresenskii and Kunyavskii \cite{VK84} (holds for $G\simeq S_n$ by $H^1(k,{\rm Pic}\, \overline{X})=0$), 
Kunyavskii \cite{Kun84} $(n=4$ (holds except for $G\simeq V_4, A_4))$, 
Drakokhrust and Platonov \cite{DP87} $(n=6$ (holds except for $G\simeq A_4, A_5))$, 
Endo \cite{End11} (holds for $G$ whose all $p$-Sylow subgroups are cyclic (general $n$),  
see also Hoshi and Yamasaki \cite[Remark 2.7]{HY3}),  
Macedo \cite{Mac20} (holds for $G\simeq A_n$ ($n\neq 4)$ by $H^1(k,{\rm Pic}\, \overline{X})=0$), 
Macedo and Newton \cite{MN22} 
($G\simeq A_4$, $S_4$, $A_5$, $S_5$, $A_6$, $A_7$ $($general $n))$, 
Hoshi, Kanai and Yamasaki \cite{HKY22} $(n\leq 15$ $(n\neq 12))$, 
(holds for $G\simeq M_n$ ($n=11,12,22,23,24$; $5$ Mathieu groups)), 
\cite{HKY23} $(n=12)$, 
\cite{HKY25} $(n=16$ and $G$; primitive$)$, 
\cite{HKY} $(G\simeq M_{11}$, $J_1$ $($general $n))$, 
Hoshi and Yamasaki \cite{HY2} (holds for $G\simeq {\rm PSL}_2(\bF_7)$ $(n=21)$, 
${\rm PSL}_2(\bF_8)$ $(n=63)$), 
\cite{HY3} (holds for metacyclic $G$ 
with trivial Schur multiplier $M(G)=0$ (general $n$)) 
where $G={\rm Gal}(L/k)\leq S_n$ is transitive and $L/k$ is the Galois closure of $K/k$.
We also refer to 
Browning and Newton \cite{BN16} and 
Frei, Loughran and Newton \cite{FLN18}. 

The following is the main theorem of this paper 
which gives a necessary and sufficient condition 
for the Hasse norm principle for separable field extensions $K/k$, 
i.e. the determination of the Shafarevich-Tate group $\Sha(T)$ 
of the norm one tori $T=R^{(1)}_{K/k}(\bG_{m,K})$ 
(see Ono's theorem (Theorem \ref{thOno})), 
with $[K:k]=[G:H]=p^3$ or $p^2$ 
when the Galois group $G={\rm Gal}(L/k)$ of the Galois closure $L/k$ 
of $K/k$ 
is the Heisenberg group $E_p(p^3)\simeq (C_p)^2\rtimes C_p$ of order $p^3$ 
with Schur multiplier $M(G)\simeq H^2(G,\bQ/\bZ)\simeq (\bZ/p\bZ)^{\oplus 2}$: 
%
\begin{theorem}[Hasse norm principle for Heisenberg extensions of degree $p^3$, see Theorem \ref{thmain} for the precise statement]\label{thmain0}
Let $k$ be a global field, $K/k$ be a finite 
separable field extension and $L/k$ be the Galois closure of $K/k$ 
with Galois groups $G={\rm Gal}(L/k)$ and $H={\rm Gal}(L/K)\lneq G$. 
Let $T=R^{(1)}_{K/k}(\bG_{m,K})$ be the norm one torus of $K/k$, 
$A(T)$ be the kernel of the weak approximation of $T$ and 
$\Sha(T)$ be the Shafarevich-Tate group of $T$. 
Let $p\geq 3$ be a prime number. 
Assume that 
\begin{align*}
G=E_p(p^3)&=\langle a,b,c\mid a^p=b^p=c^p=1, [b,a]=c, [c,a]=1, [c,b]=1\rangle\\
&=\langle b,c\rangle\rtimes\langle a\rangle=\langle a,c\rangle\rtimes\langle b\rangle\simeq (C_p)^2\rtimes C_p
\end{align*}
is the Heisenberg group of order $p^3$, i.e. 
the extraspecial group of order $p^3$ with exponent $p$. 
By the condition $\bigcap_{\sigma\in G}H^\sigma=\{1\}$, we have $H=\{1\}$ or 
$H\simeq C_p$ with $H\neq Z(G)=\langle c\rangle$. 
Then\\
{\rm (1)} When $H=\{1\}$, 
we have $H^2(G,J_G)=\Sha^2_\omega(G,J_G)\simeq H^2(G,\bQ/\bZ)\simeq 
(\bZ/p\bZ)^{\oplus 2}\geq \Sha(T)^\vee$ 
where $\Sha(T)^\vee={\rm Hom}(\Sha(T),\bQ/\bZ)$ 
and\\ 
{\rm (I)} $\Sha(T)=0$ 
if and only if $A(T)\simeq (\bZ/p\bZ)^{\oplus 2}$ 
if and only if 
{\rm (i)} there exist $($ramified$)$ places $v_1$, $v_2$ of $k$ such that 
$(C_p)^2\simeq ((C_p)^2)^{(i)}\leq G_{v_1}$, 
$(C_p)^2\simeq ((C_p)^2)^{(j)}\leq G_{v_2}$ $(1\leq i<j\leq p+1)$ 
where there exist $p+1$ non-conjugate $(C_p)^2\simeq ((C_p)^2)^{(j)}\leq G$ $(1\leq j\leq p+1)$ 
or {\rm (ii)} there exists a $($ramified$)$ place $v$ of $k$ such that $G_v=G$;\\
{\rm (II)} $\Sha(T)\simeq \bZ/p\bZ$ 
if and only if $A(T)\simeq \bZ/p\bZ$ 
if and only if there exists a $($ramified$)$ place $v$ of $k$ such that 
$(C_p)^2\leq G_v$ and {\rm (I)} does not hold;\\
{\rm (III)} $\Sha(T)\simeq (\bZ/p\bZ)^{\oplus 2}$ 
if and only if $A(T)=0$ 
if and only if $G_v\leq C_p$ for any place $v$ of $k$.\\
{\rm (2)} When $H\simeq C_p$ with $H\neq Z(G)=\langle c\rangle$, 
we have $H^2(G,J_{G/H})\simeq  H^2(G,\bQ/\bZ)
\simeq  (\bZ/p\bZ)^{\oplus 2}\geq 
\Sha^2_\omega(G,J_{G/H})
\simeq \bZ/p\bZ\geq\Sha(T)^\vee$ 
where $\Sha(T)^\vee={\rm Hom}(\Sha(T),\bQ/\bZ)$ 
and\\
{\rm (I)} $\Sha(T)=0$ 
if and only if $A(T)\simeq \bZ/p\bZ$ 
if and only if there exists a $($ramified$)$ place $v$ of $k$ 
such that $(C_p)^2\leq G_v$;\\
{\rm (II)} $\Sha(T)\simeq\bZ/p\bZ$ 
if and only if $A(T)=0$ 
if and only if $G_v\leq C_p$ for any place $v$ of $k$. 
\end{theorem}
We note that 
the case where $p=3$ of Theorem \ref{thmain0} (2) is obtained 
in Hoshi, Kanai and Yamasaki \cite[Theorem 1.18, Table 2 for $G=9T7$]{HKY22} 
by using GAP \cite{GAP}.


We organize this paper as follows. 
In Section \ref{S2}, we prove Theorem \ref{thmain0}. 
Indeed, we give a proof of Theorem \ref{thmain} 
which is a more detailed version of Theorem \ref{thmain0}. 
As an application of Theorem \ref{thmain}, 
we also get the Tamagawa number $\tau(T)$ of the corresponding 
norm one tori $T=R^{(1)}_{K/k}(\bG_{m,K})$. 
%

\section{Proof of Theorem \ref{thmain0} (Theorem \ref{thmain})}\label{S2}
Let $K/k$ be a separable field extension with $[K:k]=n$ 
and $L/k$ be the Galois closure of $K/k$. 
Let $G={\rm Gal}(L/k)$ and $H={\rm Gal}(L/K)\lneq G$ with $[G:H]=n$. 
Then we have $\bigcap_{\sigma\in G} H^\sigma=\{1\}$ 
where $H^\sigma=\sigma^{-1}H\sigma$ and hence 
$H$ contains no normal subgroup of $G$ except for $\{1\}$. 
The Galois group $G$ may be regarded as a transitive subgroup of 
the symmetric group $S_n$ of degree $n$. 
We may assume that 
$H$ is the stabilizer of one of the letters in $G\leq S_n$, 
i.e. $L=k(\theta_1,\ldots,\theta_n)$ and $K=L^H=k(\theta_i)$ 
where $1\leq i\leq n$. 

Let $Z(G)$ be the center of $G$, 
$[a,b]:=a^{-1}b^{-1}ab$ be the commutator of $a,b\in G$, 
$[G,G]:=\langle [a,b]\mid a,b\in G\rangle$ be the commutator subgroup of $G$ 
and $G^{ab}:=G/[G,G]$ be the abelianization of $G$, i.e. the maximal abelian quotient of $G$. 
Let $M(G)=H^2(G,\bC^\times)\simeq H^2(G,\bQ/\bZ)\xrightarrow[\sim]{\delta} 
H^3(G,\bZ)$ be the Schur multiplier of $G$ 
where $\delta$ is the connecting homomorphism 
(see e.g. Neukirch, Schmidt and Wingberg \cite[Chapter I, \S 3, page 26]{NSW00}). 

The following proposition is useful to get $H^2(G,J_{G/H})\geq \Sha^2_\omega(G,J_{G/H})$: 
\begin{proposition}[{Hoshi and Yamasaki \cite[Proposition 4.1]{HY3}}]\label{prop2.1}
Let $k$ be a field, $K/k$ be a finite 
separable field extension and $L/k$ be the Galois closure of $K/k$ with 
Galois groups $G={\rm Gal}(L/k)$ and $H={\rm Gal}(L/K)\lneq G$. 
Then we have an exact sequence 
\begin{align*}
H^1(G,\bQ/\bZ)\simeq G^{ab}\xrightarrow{\rm res} 
H^1(H,\bQ/\bZ)\simeq H^{ab}\to H^2(G,J_{G/H})\xrightarrow{\delta} 
H^3(G,\bZ)\simeq M(G)\xrightarrow{\rm res} H^3(H,\bZ)\simeq M(H)
\end{align*}
where $\delta$ is the connecting homomorphism.\\ 
{\rm (1)} 
$H^1(G,\bQ/\bZ)\simeq G^{ab}
\xrightarrow{\rm res}H^1(H,\bQ/\bZ)\simeq H^{ab}$ 
is surjective if and only if $[G,G]\cap H=[H,H]$;\\
{\rm (2)} 
If $[G,G]\cap H=[H,H]$ and $M(G)=0$, 
then $\Sha^2_\omega(G,J_{G/H})\leq H^2(G,J_{G/H})$ $\xrightarrow[\sim]{\delta}$ $H^3(G,\bZ)\simeq M(G)=0$;\\ 
{\rm (3)} 
If $[G,G]\cap H=[H,H]$ and $M(H)=0$, 
then $\Sha^2_\omega(G,J_{G/H})\leq H^2(G,J_{G/H})$ $\xrightarrow[\sim]{\delta}$ $H^3(G,\bZ)\simeq M(G)$;\\
{\rm (4)} If there exists $H^\prime$ $\lhd$ $G$ such that 
$G/H^\prime$ is abelian and $H^\prime$ $\cap$ $H=\{1\}$, then 
$[G,G]\cap H=[H,H]$. 
In particular, 
if $M(G)=0$ $($resp. $M(H)=0$$)$, then 
$\Sha^2_\omega(G,J_{G/H})\leq H^2(G,J_{G/H})$ $\xrightarrow[\sim]{\delta}$ $H^3(G,\bZ)\simeq M(G)=0$  
$($resp.  $\Sha^2_\omega(G,J_{G/H})\leq H^2(G,J_{G/H})$ $\xrightarrow[\sim]{\delta}$ $H^3(G,\bZ)\simeq M(G)$$)$. 

When $k$ is a global field, $\Sha^2_\omega(G,J_{G/H})=0$ implies that 
$A(T)=0$, i.e. 
$T$ has the weak approximation property, 
and $\Sha(T)=0$, i.e. 
Hasse principle holds for all torsors $E$ under $T$, 
where $T=R^{(1)}_{K/k}(\bG_{m,K})$ is the norm one torus of $K/k$ 
and hence the Hasse norm principle holds for $K/k$ 
$($see Section \ref{S1} and Ono's theorem $($Theorem \ref{thOno}$))$. 
\end{proposition}
\begin{proof}
We include a proof for the convenience of the reader. 

By the definition, we have $0\to \bZ\to \bZ[G/H]\to J_{G/H}\to 0$ where 
$J_{G/H}\simeq \widehat{T}={\rm Hom}(T,\bG_{m,L})$ 
and $T=R^{(1)}_{K/k}(\bG_{m,K})$ is the norm one torus of $K/k$. 
Then we get 
\begin{align*}
H^2(G,\bZ)\to H^2(G,\bZ[G/H])\to H^2(G,J_{G/H})
\xrightarrow{\delta} H^3(G,\bZ)\to H^3(G,\bZ[G/H]). 
\end{align*} 
Then $H^2(G,\bZ)\simeq H^1(G,\bQ/\bZ)
={\rm Hom}(G,\bQ/\bZ)\simeq G^{ab}=G/[G,G]$. 
We have 
$H^2(G,\bZ[G/H])\simeq H^2(H,\bZ)\simeq H^{ab}$ 
and $H^3(G,\bZ[G/H])\simeq H^3(H,\bZ)\simeq M(H)$ by Shapiro's lemma 
(see e.g. Brown \cite[Proposition 6.2, page 73]{Bro82}, Neukirch, Schmidt and Wingberg \cite[Proposition 1.6.3, page 59]{NSW00}). 
This implies that 
\begin{align*}
H^1(G,\bQ/\bZ)\simeq G^{ab}\xrightarrow{\rm res} 
H^1(H,\bQ/\bZ)\simeq H^{ab}\to H^2(G,J_{G/H})\xrightarrow{\delta} 
H^3(G,\bZ)\simeq M(G)\xrightarrow{\rm res} H^3(H,\bZ)\simeq M(H). 
\end{align*}
(1) We see that ${\rm Image}\{H^1(G,\bQ/\bZ)\simeq G^{ab}=G/[G,G]\xrightarrow{\rm res} 
H^1(H,\bQ/\bZ)\simeq H^{ab}=H/[H,H]\}=H/([G,G]\cap H)$. 
Hence 
$H^1(G,\bQ/\bZ)\simeq G^{ab}
\xrightarrow{\rm res}H^1(H,\bQ/\bZ)\simeq H^{ab}$ 
is surjective if and only if $[G,G]\cap H=[H,H]$.\\
(2) If $[G,G]\cap H=[H,H]$, then by (1) we have 
$H^2(G,\bZ)\simeq G^{ab}\xrightarrow{\rm res}H^2(H,\bZ)\simeq H^{ab}$ 
is surjective. 
This implies that $H^2(G,J_{G/H})\xrightarrow{\delta}H^3(G,\bZ)\simeq M(G)=0$ becomes an isomorphism.\\ 
(3) 
If $[G,G]\cap H=[H,H]$, then by (1) we have 
$H^2(G,\bZ)\simeq G^{ab}\xrightarrow{\rm res}H^2(H,\bZ)\simeq H^{ab}$ 
is surjective. 
Hence if we also have $M(H)=0$, 
then $H^2(G,J_{G/H})\xrightarrow{\delta}H^3(G,\bZ)\simeq M(G)$ 
becomes an isomorphism.\\
(4) Because $G/H^\prime$ is abelian, we have $[G,G]\leq H^\prime$. 
It follows from $H^\prime\cap H=\{1\}$ that $[G,G]\cap H=\{1\}$. 
Hence if we take $h\in H$, then 
$\widetilde{h}\in H/(H\cap [G,G])\simeq H[G,G]/[G,G]
\leq G^{ab}=G/[G,G]$ maps to 
$\overline{h}\in H^{ab}=H/[H,H]$. 
This implies that 
$G^{ab}\xrightarrow{\rm res}H^{ab}$ is surjective. 
The last assertion of (4) follows from (2), (3). 

Because $\Sha^2_\omega(G,J_{G/H})\leq H^2(G,J_{G/H})$, 
when $k$ is a global field, 
by Theorem \ref{thV} and Theorem \ref{thCTS87}, we have 
\begin{align*}
H^2(G,J_{G/H})=0\ \Rightarrow\ \Sha^2_\omega(G,J_{G/H})=0\ \Rightarrow\ A(T)=0\ {\rm and}\ \Sha(T)=0
\end{align*}
where $T=R^{(1)}_{K/k}(\bG_{m,K})$ is the norm one torus of $K/k$ 
with $\widehat{T}\simeq J_{G/H}$. 
In particular, it follows from Ono's theorem (Theorem \ref{thOno}) that 
$\Sha(T)=0$ if and only if the Hasse norm principle holds for $K/k$. 
\end{proof}

A $p$-group $G$ is called {\it special} if 
either (i) $G$ is elementary abelian; or 
(ii) $Z(G)=[G,G]=\Phi(G)$ and $Z(G)$ is elementary abelian 
where $Z(G)$ is the center of $G$, 
$[G,G]$ is the commutator subgroup of $G$ and 
$\Phi(G)$ is the Frattini subgroup of $G$, i.e. 
the intersection of all maximal subgroups of $G$ 
(see e.g. Suzuki \cite[Definition 4.14, page 67]{Suz86}). 
It follows from $[G,G]=\Phi(G)$ that $G^{ab}=G/[G,G]$ 
is elementary abelian (see Karpilovsky \cite[Lemma 3.1.2, page 114]{Kar87}). 
A $p$-group $G$ is called {\it extraspecial} if 
$G$ is non-abelian, special and $Z(G)$ is cyclic. 
We see that $Z(G)\simeq C_p$. 
If $G$ is an extraspecial $p$-group, then $|G|=p^{2n+1}$ 
and there exist exactly $2$ extraspecial groups of order $p^{2n+1}$. 
For odd prime $p\geq 3$, they are 
$E_p(p^{2n+1})$, $E_{p^2}(p^{2n+1})$  
of order $p^{2n+1}$ with exponent $p$ and $p^2$ respectively 
(see 
Gorenstein \cite[Section 5.3, page 183]{Gor80}, 
Suzuki \cite[Definition 4.14, page 67]{Suz86}). 
Note that $D_4$ and $Q_8$ are extraspecial groups of order $8$ 
with exponent $4$. 
\begin{proposition}\label{prop2.2}
Let $k$ be a field, $K/k$ be a finite 
separable field extension and $L/k$ be the Galois closure of $K/k$ with 
Galois groups $G={\rm Gal}(L/k)$ and $H={\rm Gal}(L/K)\lneq G$.
Let $p\geq 3$ be a prime number. 
Assume that 
\begin{align*}
G=E_p(p^3)&=\langle a,b,c\mid a^p=b^p=c^p=1, [b,a]=c, [c,a]=1, [c,b]=1\rangle\\
&=\langle b,c\rangle\rtimes\langle a\rangle=\langle a,c\rangle\rtimes\langle b\rangle\simeq (C_p)^2\rtimes C_p
\end{align*}
is the Heisenberg group of order $p^3$, i.e. 
the extraspecial group of order $p^3$ with exponent $p$. 
By the condition $\bigcap_{\sigma\in G}H^\sigma=\{1\}$, we have $H=\{1\}$ or $
H\simeq C_p$ with $H\neq Z(G)=\langle c\rangle\simeq C_p$. 
Then\\
{\rm (1)} When $H=\{1\}$, we have 
$\Sha^2_\omega(G,J_G)=H^2(G,J_{G})\xrightarrow[\sim]{\delta} H^3(G,\bZ)\simeq M(G)\simeq (\bZ/p\bZ)^{\oplus 2}$.\\
{\rm (2)} When $H\simeq C_p$ with $H\neq Z(G)=\langle c\rangle\simeq C_p$, 
$H^2(G,J_{G/H})\xrightarrow[\sim]{\delta} H^3(G,\bZ)\simeq M(G)\simeq (\bZ/p\bZ)^{\oplus 2}$.
\end{proposition}
\begin{proof}
We first see that $M(G)\simeq(\bZ/p\bZ)^{\oplus 2}$ 
(see Beyl and Tappe \cite[Corollary 4.16 (b), page 223]{BT82}, see also Karpilovsky \cite[Theorem 3.3.6, page 138]{Kar87}).\\
(1) It follows from Theorem \ref{thV2} and Theorem \ref{thCTS87}. 
This also follows from Proposition \ref{prop2.1} (3).\\
(2) There exist exactly $p+1$ subgroups $H\simeq C_p\leq G$ 
with $H\neq Z(G)=\langle c\rangle\simeq C_p$ up to conjugacy: 
$H=H_i=\langle ab^i\rangle\simeq C_p$ $(0\leq i\leq p-1)$, 
$H=H_{p}=\langle b\rangle\simeq C_p$ $\leq G$. 
Then it follows from Proposition \ref{prop2.1} (4) that 
$H^2(G,J_{G/H})\xrightarrow[\sim]{\delta} H^3(G,\bZ)\simeq M(G)$ 
is an isomorphism because we can take $H^\prime\lhd G$ such that 
$G/H^\prime$ is abelian and $H^\prime\cap H=\{1\}$ 
as $H^\prime=\langle b,c\rangle$ (resp. $H^\prime=\langle a,c\rangle$) 
for $H=H_i=\langle ab^i\rangle$ $(0\leq i\leq p-1)$ 
(resp. $H=H_p=\langle b\rangle$). 
\end{proof}


In order to get 
$\Sha(T)^\vee=\Sha^1(G,T)^\vee\simeq 
\Sha^2(G,\widehat{T})=\Sha^2(G,J_{G/H})$
where 
\begin{align*}
\Sha^2(G,J_{G/H})&={\rm Ker}\left\{H^2(G,J_{G/H})\xrightarrow{\rm res} \bigoplus_{v\in V_k} 
H^2(G_v,J_{G/H})\right\}\\
&=\bigcap_{v\in V_k}{\rm Ker}\left\{H^2(G,J_{G/H})\xrightarrow{\rm res}H^2(G_v,J_{G/H})\right\}, 
\end{align*}
we should analyze 
\begin{align*}
{\rm Ker}\left\{H^2(G,J_{G/H})\xrightarrow{\rm res} 
H^2(H^\prime,J_{G/H})\right\}
\end{align*}
for any $H^\prime\leq G$ beyond $\Sha^2_\omega(G,J_{G/H})$ 
where 
\begin{align*}
\Sha^2_\omega(G,J_{G/H})
&={\rm Ker}\left\{H^2(G,J_{G/H})\xrightarrow{{\rm res}}\bigoplus_{C_n\leq G:{\rm\, cyclic}}H^2(C_n,J_{G/H})\right\}\\
&=\bigcap_{C_n\leq G:{\rm\, cyclic}}{\rm Ker}\left\{H^2(G,J_{G/H})\xrightarrow{\rm res}H^2(C_n,J_{G/H})\right\} 
\end{align*} 
(see Section \ref{S1}). 
Note that if $v\in V_k$ is unramified, then $G_v\simeq C_n$. 
Hence we should understand the kernel of the restriction map 
$H^2(G,J_{G/H})\xrightarrow{\rm res}H^2(G_v,J_{G/H})$ 
for ramified places $v\in V_k$ with 
$G_v=H^\prime\not\simeq C_n$.\\


There exists a one-to-one correspondence between 
$H^2(G,A)$ and the set $M(G,A)$ of equivalence classes of extensions of $G$ by $A$ 
(see e.g. Hilton and Stammbach \cite[VI.10, page 206]{HS71}, 
Brown \cite[IV.3, page 91]{Bro82}, 
Karpilovsky \cite[Section 2.3, page 38]{Kar87}). 

Let $A$ be a $G$-module. 
For an extension $0 \to A\xrightarrow{i} E\xrightarrow{\pi} G\to 1$, 
we choose a set-theoretic section $s$ of $\pi$, i.e. a map $s:G\to E$, 
$g\mapsto \overline{g}$ with $\pi s={\rm id}_G$. 
We see that $0\in H^2(G,A)$ corresponds to the semidirect product $E=i(A)\rtimes G$ 
and 
the case where the $G$-action on $A$ is trivial corresponds 
to the central extensions of $G$ by $A$, i.e. $i(A)\leq Z(G)$ 
(see Hilton and Stammbach \cite[VI.10, page 206]{HS71}, Brown \cite[IV.3, page 92]{Bro82}). 

We can define a map $f:G\times G\to A$ by 
\begin{align*}
\overline{g}\overline{h}=i(f(g,h))\overline{gh}.
\end{align*}
For simplicity, we assume that $s(1)=\overline{1}=1$. 
Then this implies $f(g,1)=f(1,g)=0$ for any $g\in G$. 
We find that 
it follows from $(\overline{g_1}\,\overline{g_2})\overline{g_3}=\overline{g_1}(\overline{g_2}\,\overline{g_3})$ 
that 
$f$ is a $2$-cocycle: 
$f(g_1,g_2)+f(g_1g_2,g_3)=g_1f(g_2,g_3)+f(g_1,g_2g_3)$ for any $g_1,g_2,g_3\in G$.
Then $f$ is a $2$-coboundary if and only if there exists a map $\varphi: G\to A$ 
such that $f(g,h)=g\varphi(h)-\varphi(gh)+\varphi(g)$ 
(see Hilton and Stammbach \cite[VI.10, page 206]{HS71}, Brown \cite[IV.3, page 92]{Bro82}). 

From now on, for a prime number $p\geq 3$, 
we take the Heisenberg group 
\begin{align*}
G=E_p(p^3)&=\langle a,b,c\mid a^p=b^p=c^p=1, [b,a]=c, [c,a]=1, [c,b]=1\rangle\\
&=\langle b,c\rangle\rtimes\langle a\rangle=\langle a,c\rangle\rtimes\langle b\rangle\simeq (C_p)^2\rtimes C_p
\end{align*}
of order $p^3$, i.e. 
the extraspecial group of order $p^3$ with exponent $p$. 
We consider central extensions of $G$ by $A=\bZ/p\bZ=\langle z\rangle$: 
\begin{align*}
1\to \bZ/p\bZ=\langle z\rangle\xrightarrow{i} \overline{G}\, 
\mathrel{\mathop{\rightleftarrows}^{\pi}_{s}}\, G\to 1.
\end{align*}
We take $i(z)=\oz\in \overline{G}$. 
We also take $\oa=s(a)$, $\ob=s(b)$ with $\oa^p=\ob^p=\overline{1}$ 
and define $\oc=s(c):=[\ob,\oa]\in \overline{G}$. 
Then $\pi(\oc)=c$, $\oc^p=[\ob,\oa^p]([\oc,\oa]^p)^{-\frac{(p-1)}{2}}=[\ob,\overline{1}](\overline{1})^{-\frac{(p-1)}{2}}=\overline{1}$. 
Any element $g\in G$ can be written as 
$g=a^sb^tc^u$ $(0\leq s,t,u\leq p-1)$ 
by using commutator relations $[b,a]=c$, $[c,a]=1$, $[c,b]=1$. 
Define $\overline{a^sb^tc^u}:=\oa^s\ob^t\oc^u$. 
Then we have the following $p^2$ possibilities of central extensions $\overline{G}$ of $G$ by $A$
with $0\leq l,m\leq p-1$: 
\begin{align*}
\overline{G}=\langle \oa,\ob,\oc,\oz\mid 
\oa^p=\ob^p=\oc^p=1, [\ob,\oa]=\oc, [\oc,\oa]=\oz^l, [\oc,\ob]=\oz^m\rangle=
(\langle\oc,\oz\rangle\rtimes\langle\ob\rangle)\rtimes\langle \oa\rangle. 
\end{align*}
Indeed, we see that $H^2(G,\bZ/p\bZ)\simeq H^2(G,\bQ/\bZ)\simeq M(G)\simeq (\bZ/p\bZ)^{\oplus 2}$ 
(see Karpilovsky \cite[Theorem 3.3.6, page 138]{Kar87}). 
Then we may get the corresponding $2$-cocycles $f_1^l f_2^m\in Z^2(G,\bQ/\bZ)$ $(0\leq l,m\leq p-1)$ 
with $H^2(G,\bQ/\bZ)=\langle \overline{f_1}, \overline{f_2}\rangle\simeq (\bZ/p\bZ)^{\oplus 2}$ 
as in Lemma \ref{lem2.3} ($f_1, f_2$ correspond to $(l,m)=(1,0),(0,1)$ respectively):  
\begin{lemma}[{Packer \cite[Proposition 1.1]{Pak87}, see also Hatui and Singla \cite[Lemma 2.2]{HS21}}]\label{lem2.3}
Let $p\geq 3$ be a prime number and 
\begin{align*}
G=E_p(p^3)&=\langle a,b,c\mid a^p=b^p=c^p=1, [b,a]=c, [c,a]=1, [c,b]=1\rangle\\
&=\langle b,c\rangle\rtimes\langle a\rangle=\langle a,c\rangle\rtimes\langle b\rangle\simeq (C_p)^2\rtimes C_p
\end{align*}
be the Heisenberg group of order $p^3$, i.e. 
the extraspecial group of order $p^3$ with exponent $p$. 
Then $H^2(G,\bQ/\bZ)=\langle \overline{f_1}, \overline{f_2}\rangle\simeq (\bZ/p\bZ)^{\oplus 2}$ where 
$f_1,f_2\in Z^2(G,\bQ/\bZ)$ are $2$-cocycles of $G$ with coefficients $\bQ/\bZ$ which are given by 
\begin{align*}
f_1:&\ G\times G\to \bQ/\bZ, (a^{s_1}b^{t_1}c^{u_1},a^{s_2}b^{t_2}c^{u_2})\mapsto 
\frac{u_1s_2+t_1\frac{s_2(s_2-1)}{2}}{p},\\ 
f_2:&\ G\times G\to \bQ/\bZ, (a^{s_1}b^{t_1}c^{u_1},a^{s_2}b^{t_2}c^{u_2})\mapsto 
\frac{\frac{t_1(t_1-1)}{2}s_2+(t_1s_2+u_1)t_2}{p}
\end{align*}
where any element $g\in G$ can be written as 
$g=a^sb^tc^u$ $(0\leq s,t,u\leq p-1)$. 
\end{lemma}
\begin{proof}
We give a proof for the convenience of the reader. 

First we note that any element $g\in G$ can be written as 
$g=a^sb^tc^u$ $(0\leq s,t,u\leq p-1)$ 
by using commutator relations $[b,a]=c$, $[c,a]=1$, $[c,b]=1$. 

For $g_1=a^{s_1}b^{t_1}c^{u_1}$, 
$g_2=a^{s_2}b^{t_2}c^{u_2}$, 
we should compute  
$\overline{g_1}$ $\overline{g_2}=i(f(g_1,g_2))\overline{g_1g_2}$ 
with $i(f(g_1,g_2))\in\langle \oz\rangle$ and $\overline{g_1g_2}\in \langle\oa,\ob,\oc\rangle$.

When $(l,m)=(1,0)$, we have
\begin{align*}
\overline{a^{s_1}b^{t_1}c^{u_1}} \;
\overline{a^{s_2}b^{t_2}c^{u_2}}=
\overline{a}^{s_1}\overline{b}^{t_1}\overline{c}^{u_1}
\overline{a}^{s_2}\overline{b}^{t_2}\overline{c}^{u_2}
=\overline{a}^{s_1}\overline{b}^{t_1}\overline{a}^{s_2}
\overline{c}^{u_1}\overline{b}^{t_2}\overline{c}^{u_2}
\overline{z}^{u_1s_2}
=\overline{a}^{s_1}\overline{b}^{t_1}\overline{a}^{s_2}
\overline{b}^{t_2}\overline{c}^{u_1+u_2}
\overline{z}^{u_1s_2}.
\end{align*}
If $s_2=0$, then 
\begin{align*}
\overline{a^{s_1}b^{t_1}c^{u_1}} \;
\overline{a^{s_2}b^{t_2}c^{u_2}}
=\overline{a}^{s_1}\overline{b}^{t_1+t_2}\overline{c}^{u_1+u_2}
\overline{z}^{0}.
\end{align*}
If $s_2\geq 1$, then it follows from 
$\ob^{t_1}\oa^{s_2}=\oa^{s_2}\ob^{t_1}\oc^{t_1s_2}\oz^{t_1\frac{s_2}{2}(s_2-1)}$ that 
\begin{align*}
\overline{a^{s_1}b^{t_1}c^{u_1}} \;
\overline{a^{s_2}b^{t_2}c^{u_2}}
=\overline{a}^{s_1+s_2}\overline{b}^{t_1+t_2}\overline{c}^{u_1+u_2+t_1s_2}
\overline{z}^{u_1s_2+t_1\frac{s_2(s_2-1)}{2}}. 
\end{align*}
Hence we have 
$i(f(g_1,g_2))=\overline{z}^{u_1s_2+t_1\frac{s_2(s_2-1)}{2}}$, 
$\overline{g_1g_2}=\overline{a}^{s_1+s_2}\overline{b}^{t_1+t_2}\overline{c}^{u_1+u_2+t_1s_2}$. 
Therefore, we get $f_1=\varphi\circ f: G\times G\to\bQ/\bZ$ 
where $\varphi:\langle z\rangle\to \bZ/p\bZ$, $z^i\to \frac{i}{p}$ 
which corresponds to $(l,m)=(1,0)$. 

When $(l,m)=(0,1)$, we have 
\begin{align*}
\overline{a^{s_1}b^{t_1}c^{u_1}} \;
\overline{a^{s_2}b^{t_2}c^{u_2}}=
\overline{a}^{s_1}\overline{b}^{t_1}\overline{c}^{u_1}
\overline{a}^{s_2}\overline{b}^{t_2}\overline{c}^{u_2}
=\overline{a}^{s_1}\overline{b}^{t_1}\overline{a}^{s_2}
\overline{c}^{u_1}\overline{b}^{t_2}\overline{c}^{u_2}
=\overline{a}^{s_1}\overline{b}^{t_1}\overline{a}^{s_2}
\overline{b}^{t_2}\overline{c}^{u_1+u_2}
\overline{z}^{u_1t_2}. 
\end{align*}
If $t_1=0$, then 
\begin{align*}
\overline{a^{s_1}b^{t_1}c^{u_1}} \;
\overline{a^{s_2}b^{t_2}c^{u_2}}
=\overline{a}^{s_1+s_2}
\overline{b}^{t_2}\overline{c}^{u_1+u_2}
\overline{z}^{u_1t_2}. 
\end{align*}
If $t_1 \geq 1$, then it follows from 
$\ob^{t_1}\oa^{s_2}=\oa^{s_2}\ob^{t_1}\oc^{t_1s_2}\oz^{\frac{t_1}{2}(t_1-1)s_2}$ that 
\begin{align*}
\overline{a^{s_1}b^{t_1}c^{u_1}} \;
\overline{a^{s_2}b^{t_2}c^{u_2}}
=\overline{a}^{s_1+s_2}\overline{b}^{t_1+t_2}\overline{c}^{t_1s_2+u_1+u_2}
\overline{z}^{\frac{t_1(t_1-1)}{2}s_2+(t_1s_2+u_1)t_2}. 
\end{align*}
Hence we have 
$i(f(g_1,g_2))=\overline{z}^{\frac{t_1(t_1-1)}{2}s_2+(t_1s_2+u_1)t_2}$, 
$\overline{g_1g_2}=\overline{a}^{s_1+s_2}\overline{b}^{t_1+t_2}\overline{c}^{t_1s_2+u_1+u_2}$. 
Therefore, we get $f_2=\varphi\circ f: G\times G\to \bQ/\bZ$ 
where $\varphi:\langle z\rangle\to \bZ/p\bZ$, $z^i\to \frac{i}{p}$ 
which corresponds to $(l,m)=(0,1)$. 
\end{proof}
\begin{example}\label{ex2.4}
We see that 
\begin{align*}
f_1(a,a)=f_1(a,b)=f_1(a,c)=f_1(b,a)=f_1(b,b)=f_1(b,c)=f_1(c,b)=f_1(c,c)=0,\ 
f_1(c,a)=\frac{1}{p},\\
f_2(a,a)=f_2(a,b)=f_2(a,c)=f_2(b,a)=f_2(b,b)=f_2(b,c)=f_2(c,a)=f_2(c,c)=0,\ 
f_2(c,b)=\frac{1}{p}.
\end{align*}
\end{example}
\begin{lemma}\label{lem2.5}
Let $p\geq 3$ be a prime number and 
\begin{align*}
G=E_p(p^3)&=\langle a,b,c\mid a^p=b^p=c^p=1, [b,a]=c, [c,a]=1, [c,b]=1\rangle\\
&=\langle b,c\rangle\rtimes\langle a\rangle=\langle a,c\rangle\rtimes\langle b\rangle\simeq (C_p)^2\rtimes C_p
\end{align*}
be the Heisenberg group of order $p^3$, i.e. 
the extraspecial group of order $p^3$ with exponent $p$. 
Let $H^\prime\leq G$ be one of the $2p+5$ subgroups of $G$ up to conjugacy: 
\begin{align*}
&\{1\},\\ 
&Z(G)=\langle c\rangle\simeq C_p,\\
&H_0=\langle a\rangle\simeq C_p, H_1=\langle ab\rangle\simeq C_p,\ldots,  
H_{p-1}=\langle ab^{p-1}\rangle\simeq C_p, H_{p}=\langle b\rangle\simeq C_p,\\
&H_0^\prime=\langle a,c\rangle\simeq (C_p)^2, 
H_1^\prime=\langle ab,c\rangle\simeq (C_p)^2,\ldots,  
H_{p-1}^\prime=\langle ab^{p-1},c\rangle\simeq (C_p)^2, 
H_{p}^\prime=\langle b,c\rangle\simeq (C_p)^2,\\
&G=\langle a,b,c\rangle\simeq (C_p)^2\rtimes C_p. 
\end{align*}
We take $H^2(G,\bQ/\bZ)=\langle\overline{f_1},\overline{f_2}\rangle\simeq (\bZ/p\bZ)^{\oplus 2}$ 
where $\overline{f_1},\overline{f_2}$ are given as in Lemma \ref{lem2.3}. 
Then 
\begin{align*}
{\rm Ker}\{H^2(G,\bQ/\bZ)\xrightarrow{\rm res} H^2(H^\prime,\bQ/\bZ)\}=
\begin{cases}
\langle \overline{f_1},\overline{f_2}\rangle\simeq (\bZ/p\bZ)^{\oplus 2} 
& {\rm if}\quad 
H^\prime=\{1\}, Z(G)\simeq C_p, H_i\simeq C_p\ (0\leq i\leq p),\\
\langle \overline{f_2}\rangle\simeq\bZ/p\bZ 
& {\rm if}\quad H^\prime=H_{0}^\prime\simeq (C_p)^2,\\
\langle \overline{f_1}\,\overline{f_2}^{(-1)i^{-1}}\rangle\simeq \bZ/p\bZ 
& {\rm if}\quad H^\prime
=H_i^\prime\simeq (C_p)^2\ (1\leq i\leq p-1),\\
\langle \overline{f_1}\rangle\simeq\bZ/p\bZ 
& {\rm if}\quad H^\prime=H_{p}^\prime\simeq (C_p)^2,\\
0 & {\rm if}\quad H^\prime=G.
\end{cases}
\end{align*}
\end{lemma}
\begin{proof}
We note that 
$H^2(G,\bQ/\bZ)\xrightarrow{\rm res} H^2(H^\prime,\bQ/\bZ)$, 
$f\mapsto f\circ\iota$ where 
$\iota:H^\prime\times H^\prime\hookrightarrow G\times G$ is the natural injection. 
From the construction of $f_1,f_2$ as in Lemma \ref{lem2.3} and  a paragraph before Lemma \ref{lem2.3}, 
we have $H^2(G,\bQ/\bZ)=\langle \overline{f_1},\overline{f_2}\rangle\simeq (\bZ/p\bZ)^{\oplus 2}$.\\
(1) When $H^\prime=\{1\}$, $Z(G)=\langle c\rangle\simeq C_p$, $H_i=\langle ab^i\rangle\simeq C_p$ 
$(0\leq i\leq p-1)$, $H_p=\langle b\rangle\simeq C_p$. 
Because $H^2(H^\prime,\bQ/\bZ)\simeq M(H^\prime)=0$, we have 
\begin{align*}
{\rm Ker}\{H^2(G,\bQ/\bZ)\xrightarrow{\rm res} H^2(H^\prime,\bQ/\bZ)\}=
\langle \overline{f_1},\overline{f_2}\rangle\simeq (\bZ/p\bZ)^{\oplus 2}. 
\end{align*}
(2) When $H^\prime=G=\langle a,b,c\rangle$. 
Because $H^2(G,\bQ/\bZ)\xrightarrow{\rm res} H^2(G,\bQ/\bZ)$ becomes the identity map, 
we have ${\rm Ker}\{H^2(G,\bQ/\bZ)$ $\xrightarrow{\rm res}$ $H^2(H^\prime,\bQ/\bZ)\}=0$.\\
(3) When $H^\prime=H_0^\prime=\langle a,c\rangle\simeq (C_p)^2$,  
$H_p^\prime=\langle b,c\rangle\simeq (C_p)^2$. 
It follows from 
\begin{align*}
&f_1(a^{s_1}c^{u_1},a^{s_2}c^{u_2})=u_1s_2/p,\quad f_2(a^{s_1}c^{u_1},a^{s_2}c^{u_2})=0,\\ 
&f_1(b^{t_1}c^{u_1},b^{t_2}c^{u_2})=0,\quad f_2(b^{t_1}c^{u_1},b^{t_2}c^{u_2})=u_1t_2/p 
\end{align*}
for any $0\leq s_1,t_1,u_1,s_2,t_2,u_2\leq p-1$ that 
\begin{align*}
{\rm Ker}\{H^2(G,\bQ/\bZ)\xrightarrow{\rm res} H^2(H^\prime,\bQ/\bZ)\}=
\begin{cases}
\langle f_2\rangle\simeq\bZ/p\bZ 
& {\rm if}\quad H^\prime=H_{0}^\prime=\langle a,c\rangle\simeq (C_p)^2,\\
\langle f_1\rangle\simeq\bZ/p\bZ 
& {\rm if}\quad H^\prime=H_{p}^\prime=\langle b,c\rangle\simeq (C_p)^2.
\end{cases}
\end{align*}
(4) When $H^\prime=H_i^\prime=\langle ab^i,c\rangle\simeq (C_p)^2$ $(1\leq i\leq p-1)$. 
We see that 
if $[\oc,\oa]=\oz^l$, $[\oc,\ob]=\oz^m$, then $[\oc,\oa\ob^i]=\oz^{l+mi}$.
Hence $[\oc,\oa\ob^i]=1$ if and only if $m\equiv (-l) i^{-1}$ $({\rm mod}\, p)$.  
Therefore, for $l=1$, $f_1f_2^{m}\in {\rm Ker}\{H^2(G,\bQ/\bZ)\xrightarrow{\rm res} H^2(H^\prime,\bQ/\bZ)\}$ 
if and only if $m\equiv (-1) i^{-1}$ $({\rm mod}\, p)$. 
Hence we have 
\begin{align*}
{\rm Ker}\{H^2(G,\bQ/\bZ)\xrightarrow{\rm res} H^2(H_i^\prime,\bQ/\bZ)\}=
\langle f_1f_2^{(-1)i^{-1}}\rangle\simeq \bZ/p\bZ. 
\end{align*}
\end{proof}

\begin{example}[$p=3$]\label{ex2.6}
Let $G=E_p(p^3)=\langle a,b,c\rangle\simeq (C_p)^2\rtimes C_p$. 
Assume that $p=3$. 
We get
\begin{align*}
&{\rm Ker}\{H^2(G,\bQ/\bZ)\xrightarrow{\rm res} H^2(\{1\},\bQ/\bZ)\}
=H^2(G,\bQ/\bZ)=\langle f_1,f_2\rangle\simeq (\bZ/3\bZ)^{\oplus 2},\\
&{\rm Ker}\{H^2(G,\bQ/\bZ)\xrightarrow{\rm res} H^2(G,\bQ/\bZ)\}=0
\end{align*}
and\\

\renewcommand{\arraystretch}{1.2}
\begin{tabular}{c||c|cccc}
$H^\prime\simeq C_p$
& $Z(G)=\langle c\rangle$ & $H_0=\langle a\rangle$ & $H_1=\langle ab\rangle$ 
& $H_2=\langle ab^2\rangle$ & $H_3=\langle b\rangle$\\\hline
${\rm Ker}\{H^2(G,\bQ/\bZ)\xrightarrow{\rm res} H^2(H^\prime,\bQ/\bZ)\}$ 
& $(\bZ/3\bZ)^{\oplus 2}$ & $(\bZ/3\bZ)^{\oplus 2}$ & $(\bZ/3\bZ)^{\oplus 2}$ & $(\bZ/3\bZ)^{\oplus 2}$ & $(\bZ/3\bZ)^{\oplus 2}$
\end{tabular}\vspace*{2mm}

\begin{tabular}{c||cccc}
$H^\prime\simeq (C_p)^2$ & 
$H_0^\prime=\langle a,c\rangle$ & $H_1^\prime=\langle ab,c\rangle$ & 
$H_2^\prime=\langle ab^2,c\rangle$ & $H_3^\prime=\langle b,c\rangle$\\\hline
${\rm Ker}\{H^2(G,\bQ/\bZ)\xrightarrow{\rm res} H^2(H^\prime,\bQ/\bZ)\}$ 
& $\langle f_2\rangle\simeq \bZ/3\bZ$ 
& $\langle f_1f_2^2\rangle\simeq \bZ/3\bZ$ 
& $\langle f_1f_2\rangle\simeq \bZ/3\bZ$ & $\langle f_1\rangle\simeq \bZ/3\bZ$ 
\end{tabular}
\vspace*{2mm}
\renewcommand{\arraystretch}{1}
\end{example}

%
\begin{theorem}[Hasse norm principle for Heisenberg extensions of degree $p^3$: the precise statement of Theorem \ref{thmain0}]\label{thmain}
Let $k$ be a field, $K/k$ be a finite 
separable field extension and $L/k$ be the Galois closure of $K/k$ with 
Galois groups $G={\rm Gal}(L/k)$ and $H={\rm Gal}(L/K)\lneq G$. 
Let $T=R^{(1)}_{K/k}(\bG_{m,K})$ be the norm one torus of $K/k$. 
Let $p\geq 3$ be a prime number. 
Assume that  
\begin{align*}
G=E_p(p^3)&=\langle a,b,c\mid a^p=b^p=c^p=1, [b,a]=c, [c,a]=1, [c,b]=1\rangle\\
&=\langle b,c\rangle\rtimes\langle a\rangle=\langle a,c\rangle\rtimes\langle b\rangle\simeq (C_p)^2\rtimes C_p
\end{align*}
is the Heisenberg group of order $p^3$, i.e. 
the extraspecial group of order $p^3$ with exponent $p$. 
By the condition $\bigcap_{\sigma\in G}H^\sigma=\{1\}$, we have $H=\{1\}$ or 
$H\simeq C_p$ with $H\neq Z(G)=\langle c\rangle$. 
Then\\
{\rm (1)} When $H=\{1\}$, 
we have 
\begin{align*}
H^2(G,J_G)\simeq H^2(G,\bQ/\bZ)=\langle\overline{f_1},\overline{f_2}\rangle\simeq (\bZ/p\bZ)^{\oplus 2}
\end{align*} 
where $\overline{f_1},\overline{f_2}$ are given as in Lemma \ref{lem2.3} 
and 
\begin{align*}
&{\rm Ker}\{H^2(G,J_G)\xrightarrow{\rm res} H^2(H^\prime,J_G)\}\\
&\simeq 
\begin{cases}
\langle \overline{f_1},\overline{f_2}\rangle\simeq (\bZ/p\bZ)^{\oplus 2} 
& {\rm if}\quad 
H^\prime=\{1\}, Z(G)=\langle c\rangle\simeq C_p, H_i=\langle ab^i\rangle\simeq C_p\ (0\leq i\leq p-1), H_p=\langle b\rangle\simeq C_p,\\
\langle \overline{f_2}\rangle\simeq\bZ/p\bZ 
& {\rm if}\quad H^\prime=H_{0}^\prime=\langle a,c\rangle\simeq (C_p)^2,\\
\langle \overline{f_1}\,\overline{f_2}^{(-1)i^{-1}}\rangle\simeq \bZ/p\bZ 
& {\rm if}\quad H^\prime
=H_i^\prime=\langle ab^i,c\rangle\simeq (C_p)^2\ (1\leq i\leq p-1),\\
\langle \overline{f_1}\rangle\simeq\bZ/p\bZ 
& {\rm if}\quad H^\prime=H_{p}^\prime=\langle b,c\rangle\simeq (C_p)^2,\\
0 
& {\rm if}\quad H^\prime=G=\langle a,b,c\rangle\simeq (C_p)^2\rtimes C_p.
\end{cases}
\end{align*}
In particular, when $k$ is a global field, 
we have $H^2(G,J_G)=\Sha^2_\omega(G,J_G)\simeq H^2(G,\bQ/\bZ)=\langle\overline{f_1},\overline{f_2}\rangle\simeq (\bZ/p\bZ)^{\oplus 2}
\geq \Sha(T)^\vee$ where $\Sha(T)^\vee={\rm Hom}(\Sha(T),\bQ/\bZ)$ 
and\\ 
{\rm (I)} $\Sha(T)=0$ 
if and only if $A(T)\simeq (\bZ/p\bZ)^{\oplus 2}$ 
if and only if 
{\rm (i)} there exist $($ramified$)$ places $v_1$, $v_2$ of $k$ such that 
$(C_p)^2\simeq ((C_p)^2)^{(i)}\leq G_{v_1}$, 
$(C_p)^2\simeq ((C_p)^2)^{(j)}\leq G_{v_2}$ $(1\leq i<j\leq p+1)$ 
where there exist $p+1$ non-conjugate $(C_p)^2\simeq ((C_p)^2)^{(j)}\leq G$ $(1\leq j\leq p+1)$ 
or {\rm (ii)} there exists a $($ramified$)$ place $v$ of $k$ such that $G_v=G$;\\
{\rm (II)} $\Sha(T)\simeq \bZ/p\bZ$ 
if and only if $A(T)\simeq \bZ/p\bZ$ 
if and only if there exists a $($ramified$)$ place $v$ of $k$ such that 
$(C_p)^2\leq G_v$ and {\rm (I)} does not hold;\\
{\rm (III)} $\Sha(T)\simeq (\bZ/p\bZ)^{\oplus 2}$ 
if and only if $A(T)=0$ 
if and only if $G_v\leq C_p$ for any place $v$ of $k$.\\
{\rm (2)} When $H\simeq C_p$ with $H\neq Z(G)=\langle c\rangle$, 
we may assume that $H=\langle a\rangle\simeq C_p$ without loss of generality 
and we have  
\begin{align*}
H^2(G,J_{G/H})\simeq H^2(G,\bQ/\bZ)=\langle\overline{f_1},\overline{f_2}\rangle\simeq (\bZ/p\bZ)^{\oplus 2}
\end{align*} 
where $\overline{f_1},\overline{f_2}$ are given as in Lemma \ref{lem2.3} 
and 
\begin{align*}
&{\rm Ker}\{H^2(G,J_{G/H})\xrightarrow{\rm res} H^2(H^\prime,J_{G/H})\}\\
&\simeq 
\begin{cases}
\langle \overline{f_1},\overline{f_2}\rangle\simeq (\bZ/p\bZ)^{\oplus 2} 
& {\rm if}\quad 
H^\prime=\{1\}, Z(G)=\langle c\rangle\simeq C_p, H_i=\langle ab^i\rangle\simeq C_p\ (1\leq i\leq p-1), H_p=\langle b\rangle\simeq C_p,\\
\langle \overline{f_2}\rangle\simeq\bZ/p\bZ 
& {\rm if}\quad H^\prime=H=H_0=\langle a\rangle\simeq C_p,\\
0 & {\rm if}\quad H^\prime=H_{0}^\prime=\langle a,c\rangle\simeq (C_p)^2,\\
\langle \overline{f_1}\,\overline{f_2}^{(-1)i^{-1}}\rangle\simeq \bZ/p\bZ 
& {\rm if}\quad H^\prime
=H_i^\prime=\langle ab^i,c\rangle\simeq (C_p)^2\ (1\leq i\leq p-1),\\
\langle \overline{f_1}\rangle\simeq\bZ/p\bZ 
& {\rm if}\quad H^\prime=H_{p}^\prime=\langle b,c\rangle\simeq (C_p)^2,\\
0 & {\rm if}\quad H^\prime=G=\langle a,b,c\rangle\simeq (C_p)^2\rtimes C_p.
\end{cases}
\end{align*}
In particular, when $k$ is a global field, 
we have $H^2(G,J_{G/H})\simeq  H^2(G,\bQ/\bZ)=\langle\overline{f_1},\overline{f_2}\rangle\simeq  (\bZ/p\bZ)^{\oplus 2}\geq 
\Sha^2_\omega(G,J_{G/H})\simeq \langle\overline{f_2}\rangle\simeq \bZ/p\bZ\geq
\Sha(T)^\vee$ where $\Sha(T)^\vee={\rm Hom}(\Sha(T),\bQ/\bZ)$ 
and\\
{\rm (I)} $\Sha(T)=0$ 
if and only if $A(T)\simeq \bZ/p\bZ$ 
if and only if there exists a $($ramified$)$ place $v$ of $k$ 
such that $(C_p)^2\leq G_v$;\\
{\rm (II)} $\Sha(T)\simeq\bZ/p\bZ$ 
if and only if $A(T)=0$ 
if and only if $G_v\leq C_p$ for any place $v$ of $k$. 
\end{theorem}
\begin{remark}\label{rem2.8}
We note that, when $k$ is a global field, 
for a norm one torus $T=R^{(1)}_{K/k}(\bG_{m,K})$ of $K/k$, 
(i) $A(T)=0$ if and only if $T$ has the weak approximation property; 
(ii) $\Sha(T)=0$ if and only if 
Hasse principle holds for all torsors $E$ under $T$.
It also follows from Ono's theorem (Theorem \ref{thOno}) that 
$\Sha(T)=0$ if and only if Hasse norm principle holds for $K/k$ 
(see Section \ref{S1}). 
\end{remark}
\begin{proof}[Proof of Theorem \ref{thmain}]
There exist exactly the $2p+5$ subgroups $H^\prime\leq G$ up to conjugacy: 
\begin{align*}
&\{1\},\\ 
&Z(G)=\langle c\rangle\simeq C_p,\\
&H_0=\langle a\rangle, H_1=\langle ab\rangle,\ldots,  
H_{p-1}=\langle ab^{p-1}\rangle, H_{p}=\langle b\rangle\simeq C_p,\\
&H_0^\prime=\langle a,c\rangle, H_1^\prime=\langle ab,c\rangle,\ldots,  
H_{p-1}^\prime=\langle ab^{p-1},c\rangle, H_{p}^\prime=\langle b,c\rangle\simeq (C_p)^2,\\
&G=\langle a,b,c\rangle\simeq (C_p)^2\rtimes C_p. 
\end{align*}

In order to get 
\begin{align*}
\Sha^2_\omega(G,J_{G/H})&=\bigcap_{C_p\leq G:{\rm\, cyclic}}{\rm Ker}\left\{H^2(G,J_{G/H})\xrightarrow{\rm res}H^2(C_p,J_{G/H})\right\},\\
\Sha(T)^\vee\simeq \Sha^2(G,J_{G/H})&=\bigcap_{H^\prime\leq G}{\rm Ker}\left\{H^2(G,J_{G/H})\xrightarrow{\rm res}H^2(H^\prime,J_{G/H})\right\}, 
\end{align*}
we should analyze 
\begin{align*}
{\rm Ker}\left\{H^2(G,J_{G/H})\xrightarrow{\rm res}H^2(H^\prime,J_{G/H})\right\}
\end{align*} 
for each subgroup 
$H^\prime=\{1\}$, $Z(G)=\langle c\rangle$, $H_i=\langle ab^i\rangle$ $(0\leq i\leq p-1)$, $H_p=\langle b\rangle$, 
$H_i^\prime=\langle ab^i,c\rangle$ $(0\leq i\leq p-1)$, $H_p^\prime=\langle b,c\rangle$, $G=\langle a,b,c\rangle\leq G=E_p(p^3)$. 

As in the proof of Proposition \ref{prop2.1}, from the definition 
$0\to \bZ\to \bZ[G/H]\to J_{G/H}\to 0$ where 
$J_{G/H}\simeq \widehat{T}={\rm Hom}(T,\bG_{m,L})$ 
and $T=R^{(1)}_{K/k}(\bG_{m,K})$, we get 
\begin{align*}
H^2(G,\bZ)\to H^2(G,\bZ[G/H])\to H^2(G,J_{G/H})
\xrightarrow{\delta} H^3(G,\bZ)\to H^3(G,\bZ[G/H])
\end{align*} 
where $\delta$ is the connecting homomorphism. 
We have $H^2(G,\bZ)\simeq H^1(G,\bQ/\bZ)
={\rm Hom}(G,\bQ/\bZ)\simeq G^{ab}=G/[G,G]$, 
$H^2(G,\bZ[G/H])\simeq H^2(H,\bZ)\simeq H^{ab}$ 
and $H^3(G,\bZ[G/H])\simeq H^3(H,\bZ)\simeq M(H)$ by Shapiro's lemma 
(see e.g. Brown \cite[Proposition 6.2, page 73]{Bro82}, Neukirch, Schmidt and Wingberg \cite[Proposition 1.6.3, page 59]{NSW00}). 
Hence 
\begin{align*}
H^1(G,\bQ/\bZ)\simeq G^{ab}\xrightarrow{\rm res} 
H^1(H,\bQ/\bZ)\simeq H^{ab}
\to H^2(G,J_{G/H})\xrightarrow{\delta} M(G)\simeq (\bZ/p\bZ)^{\oplus 2}\xrightarrow{\rm res} M(H)=0 
\end{align*}
where $H=\{1\}$ or $H\simeq C_p$ with $H\neq Z(G)=\langle c\rangle$. 
We also see $G^{ab}\simeq (\bZ/p\bZ)^{\oplus 2}$ and $H^{ab}\simeq H$. 
It follows from Proposition \ref{prop2.2} that 
$H^2(G,J_{G/H})\xrightarrow[\sim]{\delta} H^3(G,\bZ)\simeq M(G)$ 
becomes an isomorphism. 

Because res and $\delta$ are functorial with ${\rm res}\circ \delta=\delta\circ {\rm res}$ 
(see e.g. Neukirch, Schmidt and Wingberg \cite[Proposition 1.5.2]{NSW00}), 
for any $H^\prime\leq G$, 
we have the following commutative diagram with exact horizontal lines:
\begin{align*}
\xymatrix@=30pt{
H^2(G,\bZ)\atop\simeq G^{ab}\simeq (\bZ/p\bZ)^{\oplus 2} 
\ar[r] \ar[d]_{\rm res_1} &
H^2(G,\bZ[G/H])\atop \simeq H^{ab}\simeq H \ar[r]^-0 \ar[d]_{\rm res_2} & 
H^2(G,J_{G/H}) \ar[r]^-{\delta}_-{\sim} \ar[d]_{\rm res_3} & 
H^3(G,\bZ)\atop \simeq M(G)\simeq (\bZ/p\bZ)^{\oplus 2} 
\ar[r] \ar[d]_{\rm res_4} &
H^3(G,\bZ[G/H])\atop \simeq M(H)=0 \ar[d]_{\rm res_5}
\\
H^2(H^\prime,\bZ) \ar[r] &
H^2(H^\prime,\bZ[G/H]) \ar[r] & 
H^2(H^\prime,J_{G/H}) \ar[r]^-{\delta^\prime} & 
H^3(H^\prime,\bZ) \ar[r] &
H^3(H^\prime,\bZ[G/H]). 
}
\end{align*}

In order to analyze 
\begin{align*}
{\rm Ker}\left\{H^2(G,J_{G/H})\xrightarrow{\rm res_3}H^2(H^\prime,J_{G/H})\right\},
\end{align*} 
we first see
$\delta({\rm Ker}({\rm res}_3))\leq {\rm Ker}({\rm res}_4)$. 
Hence 
if $\delta^\prime: H^2(H^\prime,J_{G/H})\to H^3(H^\prime,\bZ)$ is injective, then 
we get 
${\rm Ker}({\rm res}_3)\xrightarrow[\sim]{\delta} {\rm Ker}({\rm res}_4)$ 
where ${\rm Ker}({\rm res}_4)$ can be obtained as in Lemma \ref{lem2.5} 
via the isomorphisms $H^2(G,\bQ/\bZ)\xrightarrow[\sim]{} H^3(G,\bZ)$ 
and $H^2(H^\prime,\bQ/\bZ)\xrightarrow[\sim]{} H^3(H^\prime,\bZ)$. 

\begin{lemma}\label{lem2.9}
Let $H=\{1\}$ or $H\simeq C_p$ with $H\neq Z(G)=\langle c\rangle$. 
For $H^\prime\leq G$, 
if $H\cap x^{-1}H^\prime x=\{1\}$ for any $x\in G$, then $H^2(H^\prime,\bZ[G/H])=0$ 
and hence $\delta^\prime: H^2(H^\prime,J_{G/H})\to H^3(H^\prime,\bZ)$ is injective. 
\end{lemma}
\begin{proof}
If $H\cap x^{-1}H^\prime x=\{1\}$ 
for any $x\in G$, then the action of $H^\prime$ on $G/H$ is free 
because ${\rm Stab}_{H^\prime}(xH)=xHx^{-1}\cap H^\prime
=x(H\cap x^{-1}H^\prime x)x^{-1}=\{1\}$. 
Hence $|H^\prime\cdot(xH)|=|H^\prime xH|/|H|=|H^\prime|$ for any $x\in G$.\\
(i) if $H^\prime\simeq C_p$, then $H^2(H^\prime,\bZ[G/H])\simeq 
\widehat{H}^0(H^\prime,\bZ[G/H])=0$ because $\bZ[G/H]$ is a permutation 
$H^\prime$-lattice with no fixed points. 
Note that if $P$ is a permutation $C_p$-lattice, then 
$\widehat{H}^0(C_p,P)=P^{C_p}/N_{C_p}(P)\simeq (\bZ/p\bZ)^{\oplus s}$ 
where $s$ is the number of fixed points in a permutation basis of $P$ 
under the action of $C_p$.\\ 
(ii) if $H^\prime\simeq (C_p)^2$, then 
$H^2(H^\prime,\bZ[G/H])\simeq H^2(H^\prime,\bZ[H^\prime]^{\oplus (p/|H|)})\simeq H^2(1,\bZ)^{\oplus (p/|H|)}=0$ 
because $|H^\prime\cdot(xH)|=|H^\prime|$ for any $x\in G$ 
and the last isomorphism follows from Shapiro's lemma 
(see e.g. Brown \cite[Proposition 6.2, page 73]{Bro82}, Neukirch, Schmidt and Wingberg \cite[Proposition 1.6.3, page 59]{NSW00}). 
\end{proof}
%
\noindent 
(1) $H=\{1\}$. 
It follows from Lemma \ref{lem2.9} that  $H^2(H^\prime,\bZ[G/H])=0$ and hence 
$\delta^\prime: H^2(H^\prime,J_{G/H})\to H^3(H^\prime,\bZ)$ is injective. 
Hence we get ${\rm Ker}({\rm res}_3)\xrightarrow[\sim]{\delta} {\rm Ker}({\rm res}_4)$ 
and ${\rm Ker}({\rm res}_4)$ can be obtained as in Lemma \ref{lem2.5} 
via the isomorphisms $H^2(G,\bQ/\bZ)\xrightarrow[\sim]{} H^3(G,\bZ)$ 
and $H^2(H^\prime,\bQ/\bZ)\xrightarrow[\sim]{} H^3(H^\prime,\bZ)$. 
Namely, we have 
\begin{align*}
H^2(G,J_G)\simeq H^2(G,\bQ/\bZ)=\langle\overline{f_1},\overline{f_2}\rangle\simeq (\bZ/p\bZ)^{\oplus 2}
\end{align*} 
where $\overline{f_1},\overline{f_2}$ are given as in Lemma \ref{lem2.3} 
and 
\begin{align*}
{\rm Ker}\{H^2(G,J_G)\xrightarrow{\rm res_3} H^2(H^\prime,J_G)\}\simeq 
\begin{cases}
\langle \overline{f_1},\overline{f_2}\rangle\simeq (\bZ/p\bZ)^{\oplus 2} 
& {\rm if}\quad 
H^\prime=\{1\}, Z(G)\simeq C_p, H_i\simeq C_p\ (0\leq i\leq p),\\
\langle \overline{f_2}\rangle\simeq\bZ/p\bZ 
& {\rm if}\quad H^\prime=H_{0}^\prime\simeq (C_p)^2,\\
\langle \overline{f_1}\,\overline{f_2}^{(-1)i^{-1}}\rangle\simeq \bZ/p\bZ 
& {\rm if}\quad H^\prime
=H_i^\prime\simeq (C_p)^2\ (1\leq i\leq p-1),\\
\langle \overline{f_1}\rangle\simeq\bZ/p\bZ 
& {\rm if}\quad H^\prime=H_{p}^\prime\simeq (C_p)^2,\\
0 & {\rm if}\quad H^\prime=G.
\end{cases}
\end{align*}
In particular, when $k$ is a global field, 
we have $H^2(G,J_G)=\Sha^2_\omega(G,J_G)\simeq H^2(G,\bQ/\bZ)=\langle\overline{f_1},\overline{f_2}\rangle\simeq (\bZ/p\bZ)^{\oplus 2}
\geq \Sha(T)^\vee$ where $\Sha(T)^\vee={\rm Hom}(\Sha(T),\bQ/\bZ)$ 
and\\ 
{\rm (I)} $\Sha(T)=0$ 
if and only if $A(T)\simeq (\bZ/p\bZ)^{\oplus 2}$ 
if and only if 
{\rm (i)} there exist $($ramified$)$ places $v_1$, $v_2$ of $k$ such that 
$(C_p)^2\simeq ((C_p)^2)^{(i)}\leq G_{v_1}$, 
$(C_p)^2\simeq ((C_p)^2)^{(j)}\leq G_{v_2}$ $(1\leq i<j\leq p+1)$ 
where there exist $p+1$ non-conjugate $(C_p)^2\simeq ((C_p)^2)^{(j)}\leq G$ $(1\leq j\leq p+1)$ 
or {\rm (ii)} there exists a $($ramified$)$ place $v$ of $k$ such that $G_v=G$;\\
{\rm (II)} $\Sha(T)\simeq \bZ/p\bZ$ 
if and only if $A(T)\simeq \bZ/p\bZ$ 
if and only if there exists a $($ramified$)$ place $v$ of $k$ such that 
$(C_p)^2\leq G_v$ and {\rm (I)} does not hold;\\
{\rm (III)} $\Sha(T)\simeq (\bZ/p\bZ)^{\oplus 2}$ 
if and only if $A(T)=0$ 
if and only if $G_v\leq C_p$ for any place $v$ of $k$.\\

\noindent 
(2) $H\simeq C_p$ with $H\neq Z(G)=\langle c\rangle$. 
We may assume that 
\begin{align*}
H=H_0=\langle a\rangle
\end{align*}
because $J_{G/H}\simeq J_{\varphi(G/H)}$ as $G$-lattices for any $\varphi\in {\rm Aut}(G)$ and 
$\varphi_1(H_i)=H_{i+1}$ $(0\leq i\leq p-2)$, $\varphi_1(H_{p-1})=H_0$, $\varphi_2(H_0)=H_p$, 
where $\varphi_1: G\to G, a\mapsto ab, b\mapsto b, c\mapsto c$, 
$\varphi_2:G\to G, a\mapsto b, b\mapsto a, c\mapsto c^{-1}\in {\rm Aut}(G)$.\\ 

Case (i). 
If $H\cap x^{-1}H^\prime x=\{1\}$ for any $x\in G$, i.e.  
\begin{align*}
H^\prime\neq H=H_0=\langle a\rangle\ {\rm and}\ H^\prime\neq H^\prime_0=\langle a,c\rangle
\end{align*} 
up to conjugacy, 
then  $H^2(H^\prime,\bZ[G/H])=0$ and 
$\delta^\prime: H^2(H^\prime,J_{G/H})\to H^3(H^\prime,\bZ)$ is injective 
by Lemma \ref{lem2.9}.  
Hence we get ${\rm Ker}({\rm res}_3)\xrightarrow[\sim]{\delta} {\rm Ker}({\rm res}_4)$ 
and ${\rm Ker}({\rm res}_4)$ can be obtained as in Lemma \ref{lem2.5} 
via the isomorphisms $H^2(G,\bQ/\bZ)\xrightarrow[\sim]{} H^3(G,\bZ)$ 
and $H^2(H^\prime,\bQ/\bZ)\xrightarrow[\sim]{} H^3(H^\prime,\bZ)$ 
as in the case (1). 
Hence we get 
\begin{align*}
&{\rm Ker}\{H^2(G,J_{G/H})\xrightarrow{\rm res_3} H^2(H^\prime,J_{G/H})\}\\
&\simeq
\begin{cases}
\langle \overline{f_1},\overline{f_2}\rangle\simeq (\bZ/p\bZ)^{\oplus 2} 
& {\rm if}\quad 
H^\prime=\{1\}, Z(G)=\langle c\rangle\simeq C_p, H_i=\langle ab^i\rangle\simeq C_p\ (1\leq i\leq p-1), H_p=\langle b\rangle\simeq C_p,\\
\langle \overline{f_1}\,\overline{f_2}^{(-1)i^{-1}}\rangle\simeq \bZ/p\bZ 
& {\rm if}\quad H^\prime
=H_i^\prime=\langle ab^i,c\rangle\simeq (C_p)^2\ (1\leq i\leq p-1),\\
\langle \overline{f_1}\rangle\simeq\bZ/p\bZ 
& {\rm if}\quad H^\prime=H_{p}^\prime=\langle b,c\rangle\simeq (C_p)^2,\\
0 
& {\rm if}\quad H^\prime=G=\langle a,b,c\rangle\simeq (C_p)^2\rtimes C_p.
\end{cases}
\end{align*}\\

Case (ii).  
If $H\cap x^{-1}H^\prime x\neq \{1\}$ for some $x\in G$, i.e.  
\begin{align*}
H^\prime =H=H_0=\langle a\rangle\ {\rm or}\ H^\prime=H^\prime_0=\langle a,c\rangle
\end{align*} 
up to conjugacy. 
In order to finish the proof, we should show that 
\begin{align*}
{\rm Ker}\{H^2(G,J_{G/H})\xrightarrow{\rm res_3} H^2(H^\prime,J_{G/H})\}\simeq 
\begin{cases}
\langle \overline{f_2}\rangle\simeq\bZ/p\bZ 
& {\rm if}\quad H^\prime=H=H_0=\langle a\rangle\simeq C_p,\\
0 
& {\rm if}\quad H^\prime=H_{0}^\prime=\langle a,c\rangle\simeq (C_p)^2. 
\end{cases}
\end{align*}
We will prove it 
for $H^\prime=H_0^\prime$ in Proposition \ref{prop2.11} 
and then for $H^\prime=H$ in Proposition \ref{prop2.15}.\\

From Case (i) and Case (ii), when $k$ is a global field, 
we have $H^2(G,J_{G/H})\simeq  H^2(G,\bQ/\bZ)=\langle\overline{f_1},\overline{f_2}\rangle\simeq  (\bZ/p\bZ)^{\oplus 2}\geq 
\Sha^2_\omega(G,J_{G/H})\simeq \langle\overline{f_2}\rangle\simeq \bZ/p\bZ\geq
\Sha(T)^\vee$ where $\Sha(T)^\vee={\rm Hom}(\Sha(T),\bQ/\bZ)$ 
and\\
{\rm (I)} $\Sha(T)=0$ 
if and only if $A(T)\simeq \bZ/p\bZ$ 
if and only if there exists a $($ramified$)$ place $v$ of $k$ 
such that $(C_p)^2\leq G_v$;\\
{\rm (II)} $\Sha(T)\simeq\bZ/p\bZ$ 
if and only if $A(T)=0$ 
if and only if $G_v\leq C_p$ for any place $v$ of $k$.\\ 
%

Let $e_{i,j}=b^ic^jH$ $(1\leq i,j\leq p)$ be a $\bZ$-basis of $\bZ[G/H]$ 
with the $G$-action 
\begin{align*}
a&: e_{i,j}\mapsto e_{i,j-i},\\
b&: e_{i,j}\mapsto e_{i+1,j},\\
c&: e_{i,j}\mapsto e_{i,j+1}
\end{align*}
where the subscripts of $e_{i,j}$ should be understood modulo $p$
with a complete set $\{1,\ldots,p\}$ of representatives 
(for example, we see the action of $a$ by using the rule  $ab^ic^j=b^iac^{-i}c^{j}=b^ic^{j-i}a$). 

We note that $J_{G/H}$ is generated by $e_{i,j}=b^ic^jH$ $(1\leq i,j\leq p)$ 
with the relation $\sum_{i=1}^{p}\sum_{j=1}^{p}e_{i,j}=0$. 
Hence a $\bZ$-basis of $J_{G/H}$ is given by $e_{i,j}$ $(1\leq i,j\leq p$ and $(i,j)\neq (p,p))$. 

\begin{lemma}\label{lem2.10}
Let $H=\langle a\rangle\simeq C_p$ $\lhd$ 
$H_0^\prime=\langle a,c\rangle\simeq (C_p)^2$ $\lhd$ 
$G=E_p(p^3)=\langle a,b,c\rangle\simeq (C_p)^2\rtimes C_p$. 
Then 
$H^1(H,J_{G/H})=0$, $H^1(H^\prime_0,J_{G/H})=0$. 
\end{lemma}
\begin{proof}
Because 
$\bZ[G/H]$ is a permutation $H$-lattice with $p$ fixed elements $e_{p,j}$ $(1\leq j\leq p)$, 
we find that $J_{G/H}$ is a permutation $H$-lattice with $p-1$ fixed elements $e_{p,j}$ $(1\leq j\leq p-1)$. 
This implies that $H^1(H,J_{G/H})=0$. 

By 
the inflation-restriction exact sequence
\begin{align*}
0\to H^1(H_0^\prime/H,(J_{G/H})^H)
\xrightarrow{\rm inf} H^1(H_0^\prime,J_{G/H})
\xrightarrow{\rm res} H^1(H,J_{G/H})^{H_0^\prime/H}
\xrightarrow{\rm tr} H^2(H_0^\prime/H,(J_{G/H})^H)
\xrightarrow{\rm inf} H^2(H_0^\prime,J_{G/H})
\end{align*}
(see Hochschild and Serre \cite{HS53}) and $H^1(H,J_{G/H})=0$ as above, we have 
\begin{align*}
0\to H^1(H_0^\prime/H,(J_{G/H})^H)\xrightarrow[\sim]{\rm inf} H^1(H_0^\prime,J_{G/H})\to 0.
\end{align*}
Because $\bZ[G/H]\simeq \bigoplus_{i=1}^{p}\bZ[H_0^\prime/\langle ac^{i}\rangle]$ 
as $H_0^\prime$-lattices, we see that 
$\bZ[G/H]^H\simeq\bigoplus_{i=1}^{p-1}\langle\sum_{j=1}^{p}e_{i,j}]\rangle_\bZ
\bigoplus_{j=1}^{p} \langle e_{p,j}\rangle_\bZ$ 
as $\bZ$-lattice with the action of $H_0^\prime/H=\langle \overline{c}\rangle$: $\overline{c}=cH: e_{i,j}\mapsto e_{i,j+1}$. 
Therefore we see that 
$\bZ[G/H]^H$ is a permutation $H_0^\prime/H$-lattice with $p-1$ fixed points 
and 
$(J_{G/H})^H$ is a permutation $H_0^\prime/H$-lattice with $p-2$ fixed points. 
This implies that $H^1(H_0^\prime/H,(J_{G/H})^H)=0\xrightarrow[\sim]{\rm inf} H^1(H_0^\prime,J_{G/H})$. 
Hence we get $H^1(H_0^\prime,J_{G/H})=0$.
\end{proof}

\begin{proposition}\label{prop2.11}
Let $H=\langle a\rangle\simeq C_p$ $\lhd$ 
$H_0^\prime=\langle a,c\rangle\simeq (C_p)^2$ $\lhd$ 
$G=E_p(p^3)=\langle a,b,c\rangle\simeq (C_p)^2\rtimes C_p$. 
We have 
\begin{align*}
{\rm Ker}\{H^2(G,J_{G/H})\xrightarrow{\rm res_3}H^2(H_0^\prime,J_{G/H})\}=0.
\end{align*}
\end{proposition}
\begin{proof}
For a normal subgroup $N\lhd G$ and $G$-module $M$, 
the Lyndon-Hochschild-Serre spectral sequence gives rise to 
the $7$-term 
exact sequence
\begin{align*}
0 &\to H^1(G/N,M^N)\xrightarrow{\rm inf} H^1(G,M)
\xrightarrow{\rm res} H^1(N,M)^{G/N} 
\xrightarrow{\rm tr} H^2(G/N,M^N) \\
&\xrightarrow{\rm inf} H^2(G,M)_1 
\xrightarrow{\rho} H^1(G/N,H^1(N,M))
\xrightarrow{\lambda} H^3(G/N,M^N)
\end{align*}
where $H^2(G,M)_1:={\rm Ker}\{H^2(G,M)
\xrightarrow{\rm res} H^2(N,M)\}$ 
(see Dekimpe, Hartl and Wauters \cite{DHW12}, see also Huebschmann \cite{Hue81a}, \cite{Hue81b}). 
We apply this to the case $N=H_0^\prime$ and $M=J_{G/H}$. 
Then we get 
\begin{align*}
H^1(H_0^\prime,J_{G/H})^{G/H_0^\prime}=0 
\xrightarrow{\rm tr} H^2(G/H_0^\prime,(J_{G/H})^{H_0^\prime})
\xrightarrow[\sim]{\rm inf} H^2(G,J_{G/H})_1 
\xrightarrow{\rho} H^1(G/H_0^\prime,H^1(H_0^\prime,J_{G/H}))=0
\end{align*}
where $H^2(G,J_{G/H})_1={\rm Ker}\{H^2(G,J_{G/H})\xrightarrow{\rm res_3}H^2(H_0^\prime,J_{G/H})\}$
because $H^1(H_0^\prime,J_{G/H})=0$ by Lemma \ref{lem2.10}. 
By $\bZ[G/H]^{H_0^\prime}\simeq\bZ[G/H_0^\prime]\simeq \bZ[\langle\overline{b}\rangle]\simeq \bZ[C_p]$ 
and $(J_{G/H})^{H_0^\prime}\simeq J_{G/H_0^\prime}\simeq J_{\langle\overline{b}\rangle}\simeq J_{C_p}$, 
we have $H^2(G/H_0^\prime,(J_{G/H})^{H_0^\prime})$ $=$ $H^2(C_p,J_{C_p})=\widehat{H}^0(C_p,J_{C_p})=(J_{C_p})^{C_p}/N_{C_p}(J_{C_p})=0$ 
because $C_p\simeq \langle\overline{b}\rangle$ acts on $J_{C_p}$ with no fixed points. 
This implies that $H^2(G,J_{G/H})_1=0$. 
\end{proof}

As in the beginning of the proof, from the definition 
$0\to \bZ\to \bZ[G/H]\to J_{G/H}\to 0$ where 
$J_{G/H}\simeq \widehat{T}={\rm Hom}(T,\bG_{m,L})$ 
and $T=R^{(1)}_{K/k}(\bG_{m,K})$, we already have  
\begin{align*}
&H^1(G,\bZ[G/H])\to H^1(G,J_{G/H})\to H^2(G,\bZ)\simeq G^{ab}\simeq (\bZ/p\bZ)^{\oplus 2}\to
H^2(G,\bZ[G/H])\simeq H^{ab}\simeq \bZ/p\bZ\\
&\xrightarrow{0} H^2(G,J_{G/H})
\simeq (\bZ/p\bZ)^{\oplus 2}\xrightarrow[\sim]{\delta} H^3(G,\bZ)\simeq M(G)\simeq(\bZ/p\bZ)^{\oplus 2}\to 
H^3(G,\bZ[G/H])\simeq M(H)=0
\end{align*} 
where $\delta$ is the connecting homomorphism. 

\begin{lemma}\label{lem2.12}
Let $H=\langle a\rangle\simeq C_p$ $\lhd$ 
$H_0^\prime=\langle a,c\rangle\simeq (C_p)^2$ $\lhd$ 
$G=E_p(p^3)=\langle a,b,c\rangle\simeq (C_p)^2\rtimes C_p$. 
Then we have\\
{\rm (1)} $H^1(G,\bZ[G/H])=0$, $H^1(G,J_{G/H})\simeq \bZ/p\bZ$;\\ 
{\rm (2)} $H^2(H^\prime_0,\bZ)\simeq (\bZ/p\bZ)^{\oplus 2}$, $H^2(H,\bZ)\simeq \bZ/p\bZ$;\\
{\rm (3)} $H^3(H^\prime_0,\bZ)\simeq \bZ/p\bZ$, $H^3(H,\bZ)=0$;\\
{\rm (4)} $H^2(H^\prime_0,\bZ[G/H])\simeq (\bZ/p\bZ)^{\oplus p}$,  $H^2(H,\bZ[G/H])\simeq (\bZ/p\bZ)^{\oplus p}$;\\
{\rm (5)} $H^3(H^\prime_0,\bZ[G/H])=0$, $H^3(H,\bZ[G/H])=0$;\\
{\rm (6)} $H^2(H^\prime_0,J_{G/H})\simeq (\bZ/p\bZ)^{\oplus p-1}$,  $H^2(H,J_{G/H})=(\bZ/p\bZ)^{\oplus p-1}$.
\end{lemma}
\begin{proof}
(1) By Shapiro's lemma 
(see e.g. Brown \cite[Proposition 6.2, page 73]{Bro82}, Neukirch, Schmidt and Wingberg \cite[Proposition 1.6.3, page 59]{NSW00}), 
we have $H^1(G,\bZ[G/H])\simeq H^1(H,\bZ)\simeq {\rm Hom}(H,\bZ)=0$. 
Hence it follows from Proposition \ref{prop2.2} (2) and the exact sequence 
\begin{align*}
H^1(G,\bZ[G/H])=0\to H^1(G,J_{G/H})\to H^2(G,\bZ)\simeq G^{ab}\simeq (\bZ/p\bZ)^{\oplus 2}\to
H^2(G,\bZ[G/H])\simeq H^{ab}\simeq \bZ/p\bZ\to 0
\end{align*}
that $H^1(G,J_{G/H})\simeq\bZ/p\bZ$.\\
(2) We have $H^2(H^\prime_0,\bZ)\simeq H^1(H^\prime_0,\bQ/\bZ)\simeq {\rm Hom}(H^\prime_0,\bQ/\bZ)
\simeq (H^\prime_0)^{ab}\simeq (\bZ/p\bZ)^{\oplus 2}$, 
$H^2(H,\bZ)\simeq H^1(H,\bQ/\bZ)\simeq {\rm Hom}(H,\bQ/\bZ)\simeq H^{ab}\simeq \bZ/p\bZ$.\\
(3) $H^3(H^\prime_0,\bZ)\simeq M(H^\prime_0)\simeq \bZ/p\bZ$, $H^3(H,\bZ)\simeq M(H)=0$.\\
(4) 
We have $H^2(H,\bZ[G/H])\simeq \widehat{H}^0(H,\bZ[G/H])\simeq 
\bZ[G/H]^H/N_H(\bZ[G/H])\simeq (\bZ/|H|\,\bZ)^{\oplus p}\simeq 
(\bZ/p\bZ)^{\oplus p}$ 
because $\bZ[G/H]$ is a permutation $H$-lattice 
with $p$ fixed elements $e_{p,j}$ $(1\leq j\leq p)$.\\ 
We also have $H^2(H^\prime_0,\bZ[G/H])\simeq 
\bigoplus_{i=1}^{p} H^2(H^\prime_0,\bZ[H^\prime_0/\langle ac^{i}\rangle])
\simeq \bigoplus_{i=1}^{p}H^2(\langle ac^{i}\rangle,\bZ)
\simeq H^2(C_p,\bZ)^{\oplus p}\simeq \widehat{H}^0(C_p,\bZ)^{\oplus p}$ 
$\simeq$ $(\bZ/p\bZ)^{\oplus p}$ because
${\rm Orb}_{H_0^\prime}(xH)\simeq H_0^\prime/{\rm Stab}_{H_0^\prime}(xH)$ 
and ${\rm Stab}_{H_0^\prime}(xH)=xHx^{-1}$ for any $x\in G$  
and hence 
$\bZ[G/H]\simeq \bigoplus_{i=1}^{p}\bZ[H_0^\prime/\langle ac^{i}\rangle]$ 
as $H_0^\prime$-lattices.\\ 
(5) 
Because 
$\bZ[G/H]\simeq \bigoplus_{i=1}^p\bZ[H_0^\prime/\langle ac^i\rangle]$ as $H_0^\prime$-lattice, 
we have 
$H^3(H_0^\prime,\bZ[G/H])\simeq \bigoplus_{i=1}^p H^3(H_0^\prime,\bZ[H_0^\prime/\langle ac^i\rangle]) 
\simeq \bigoplus_{i=1}^p H^3(\langle ac^i\rangle,\bZ)\simeq \bigoplus_{i=1}^p H^1(\langle ac^i\rangle,\bZ)\simeq 
\bigoplus_{i=1}^p H^1(C_p,\bZ)=0$. 
We also see $H^3(H,\bZ[G/H])\simeq H^1(H,\bZ[G/H])$ $=0$ because $H\simeq C_p$.\\ 
(6) 
By (2), (3), (4), (5), we get the exact sequence 
\begin{align*}
&H^1(H_0^\prime,J_{G/H})=0\to H^2(H_0^\prime,\bZ)\simeq (H_0^\prime)^{ab}\simeq (\bZ/p\bZ)^{\oplus 2}\to
H^2(H_0^\prime,\bZ[G/H])\simeq (\bZ/p\bZ)^{\oplus p}\\
&\xrightarrow{} H^2(H_0^\prime,J_{G/H})\xrightarrow{\delta} H^3(H_0^\prime,\bZ)\simeq M(H_0^\prime)\simeq \bZ/p\bZ\to 
H^3(H_0^\prime,\bZ[G/H])=0. 
\end{align*} 
This implies that $H^2(H_0^\prime,J_{G/H})\simeq (\bZ/p\bZ)^{\oplus p-1}$. 
Because 
$\bZ[G/H]$ is a permutation $H$-lattice with $p$ fixed elements $e_{p,j}$ $(1\leq j\leq p)$, 
we find that $J_{G/H}$ is a permutation $H$-lattice with $p-1$ fixed elements $e_{p,j}$ $(1\leq j\leq p-1)$. 
This implies that 
$H^2(H,J_{G/H})\simeq \widehat{H}^0(H,J_{G/H})
\simeq (J_{G/H})^H/N_H(J_{G/H})\simeq (\bZ/|H|\,\bZ)^{\oplus p-1}\simeq (\bZ/p\bZ)^{\oplus p-1}$. 
\end{proof}

By Lemma \ref{lem2.12}, 
because res and $\delta$ are functorial with ${\rm res}\circ \delta=\delta\circ {\rm res}$ 
(see e.g. Neukirch, Schmidt and Wingberg \cite[Proposition 1.5.2]{NSW00}), 
we have the following commutative diagram with exact horizontal lines:
\begin{align*}
\xymatrix@=22pt{
0\ar[r] 
&
H^1(G,J_{G/H})\atop \simeq \bZ/p\bZ \ar[r]  \ar[d]_{\rm res_0} & 
H^2(G,\bZ)\atop \simeq (\bZ/p\bZ)^{\oplus 2} \ar[r] \ar[d]_{\rm res_1} &
H^2(G,\bZ[G/H])\atop \simeq \bZ/p\bZ \ar[r]^-{0} \ar[d]_{\rm res_2} & 
H^2(G,J_{G/H})\atop \simeq (\bZ/p\bZ)^{\oplus 2} \ar[r]^-{\delta}_-{\sim} \ar[d]_{\rm res_3} & 
H^3(G,\bZ)\atop \simeq (\bZ/p\bZ)^{\oplus 2} \ar[r] \ar[d]_{\rm res_4} &
H^3(G,\bZ[G/H])\atop =0 \ar[d]_{\rm res_5}
\\
&
H^1(H_0^\prime,J_{G/H})\atop =0 \ar[r] \ar[d]_{\rm res_0^\prime} & 
H^2(H_0^\prime,\bZ)\atop \simeq (\bZ/p\bZ)^{\oplus 2} \ar[r]^-{\varphi^\prime} \ar[d]_{\rm res_1^\prime}&
H^2(H_0^\prime,\bZ[G/H])\atop \simeq (\bZ/p\bZ)^{\oplus p} \ar[r]^-{\psi^\prime} \ar[d]_{\rm res_2^\prime}& 
H^2(H_0^\prime,J_{G/H})\atop\simeq (\bZ/p\bZ)^{\oplus p-1} \ar[r]^-{\delta^\prime} \ar[d]_{\rm res_3^\prime}& 
H^3(H_0^\prime,\bZ)\atop \simeq \bZ/p\bZ \ar[r] \ar[d]_{\rm res_4^\prime}&
H^3(H_0^\prime,\bZ[G/H])\atop =0\ar[d]_{\rm res_5^\prime}
\\
&
H^1(H,J_{G/H})\atop =0 \ar[r] & 
H^2(H,\bZ)\atop \simeq \bZ/p\bZ \ar[r]^-{\varphi^{\prime\prime}} &
H^2(H,\bZ[G/H])\atop \simeq (\bZ/p\bZ)^{\oplus p}  \ar[r]^-{\psi^{\prime\prime}} & 
H^2(H,J_{G/H})\atop \simeq (\bZ/p\bZ)^{\oplus p-1} \ar[r]^-{\delta^{\prime\prime}} & 
H^3(H,\bZ)\atop =0 \ar[r] &
H^3(H,\bZ[G/H])\atop =0
}
\end{align*}
where 
$H=\langle a\rangle\simeq C_p$ $\lhd$ 
$H_0^\prime=\langle a,c\rangle\simeq (C_p)^2$ $\lhd$ 
$G=E_p(p^3)=\langle a,b,c\rangle\simeq (C_p)^2\rtimes C_p$.

\begin{lemma}\label{lem2.13}
Let $H=\langle a\rangle\simeq C_p$ $\lhd$ 
$H_0^\prime=\langle a,c\rangle\simeq (C_p)^2$ $\lhd$ 
$G=E_p(p^3)=\langle a,b,c\rangle\simeq (C_p)^2\rtimes C_p$. 
Then we have\\ 
{\rm (1)} ${\rm res}_1^\prime: H^2(H_0^\prime,\bZ)\to H^2(H,\bZ)$ is surjective;\\
{\rm (2)} 
$\varphi^{\prime\prime}: H^2(H,\bZ)\to H^2(H,\bZ[G/H])$ is injective;\\
{\rm (3)} ${\rm res}_2^\prime: H^2(H_0^\prime,\bZ[G/H])\to H^2(H,\bZ[G/H])$, $f\mapsto f_p$ 
where 
$H^2(H_0^\prime,\bZ[G/H]) \simeq \bigoplus_{i=1}^p H^2(H_0^\prime,\bZ[H_0^\prime/\langle ac^i\rangle])$, 
$f\mapsto (f_1,\ldots,f_p)$ 
and hence ${\rm Ker}({\rm res}_2^\prime)=\bigoplus_{i=1}^{p-1}H^2(H_0^\prime,\bZ[H_0^\prime/\langle ac^i\rangle])$;\\
{\rm (4)} ${\rm Im}(\varphi^{\prime\prime})={\rm Im}({\rm res}_2^\prime)$;\\
{\rm (5)} ${\rm Ker}(\delta^\prime)\subset {\rm Ker}({\rm res}_3^\prime)$. 
\end{lemma}
\begin{proof}
(1) ${\rm res}_1^\prime: H^2(H_0^\prime,\bZ)\to H^2(H,\bZ)$ is surjective 
because it is just the restriction 
${\rm res}_1^\prime: H^2(H_0^\prime,\bZ)\simeq H^1(H_0^\prime,\bQ/\bZ)
={\rm Hom}(H_0^\prime,\bQ/\bZ)\simeq H_0^\prime\simeq (C_p)^2\to 
H^2(H,\bZ)\simeq H^1(H,\bQ/\bZ)={\rm Hom}(H,\bQ/\bZ)\simeq H\simeq C_p$ 
of the character module.\\
(2) It follows from 
$H^1(H,J_{G/H})=0$ 
that $\varphi^{\prime\prime}: H^2(H,\bZ)\to H^2(H,\bZ[G/H])$ is injective.\\ 
(3) The action of $H_0^\prime=\langle a,c\rangle$ on $\bZ[G/H]$ 
yields the corresponding orbit decomposition $\bZ[G/H]\simeq \bigoplus_{i=1}^p\bZ[H_0^\prime/\langle ac^i\rangle]$ 
and $H=\langle a\rangle$ for $i=p$.\\ 
(4) It follows from (1) and the commutativity 
${\rm res}_2^\prime\circ \varphi^\prime=\varphi^{\prime\prime}\circ{\rm res}_1^\prime$ 
that ${\rm Im}(\varphi^{\prime\prime})\leq {\rm Im}({\rm res}_2^\prime)$. 
By (2), we have ${\rm Im}(\varphi^{\prime\prime})\simeq \bZ/p\bZ$. 
By (3), we have ${\rm Im}({\rm res}_2^\prime)\simeq \bZ/p\bZ$. 
Hence we have ${\rm Im}(\varphi^{\prime\prime})={\rm Im}({\rm res}_2^\prime)$.\\ 
(5) 
Take $\alpha\in {\rm Ker}(\delta^\prime)$. 
By ${\rm Ker}(\delta^\prime)={\rm Im}(\psi^\prime)$, 
there exists $\beta\in H^2(H_0^\prime,\bZ[G/H])$ such that $\psi^\prime(\beta)=\alpha$. 
By (1), (4) and the commutativity 
${\rm res}_2^\prime\circ \varphi^\prime=\varphi^{\prime\prime}\circ{\rm res}_1^\prime$, 
there exists $\gamma\in H^2(H_0^\prime,\bZ)$ such that 
${\rm res}_2^\prime(\varphi^\prime(\gamma))={\rm res}_2^\prime(\beta)$. 
Hence we have $\beta-\varphi^\prime(\gamma)\in {\rm Ker}({\rm res}_2^\prime)$.  
This implies that $\alpha=\psi^\prime(\beta)=\psi^\prime(\beta-\varphi^\prime(\gamma))$ 
because $\psi^\prime(\varphi^\prime(\gamma))=0$. 
By the commutativity 
${\rm res}_3^\prime\circ \psi^\prime=\psi^{\prime\prime}\circ{\rm res}_2^\prime$, 
we get ${\rm res}_3^\prime(\alpha)={\rm res}_3^\prime\circ \psi^\prime(\beta-\varphi^\prime(\gamma))
=\psi^{\prime\prime}\circ {\rm res}_2^\prime(\beta-\varphi^\prime(\gamma))=0$ 
because $\beta-\varphi^\prime(\gamma)\in {\rm Ker}({\rm res}_2^\prime)$. 
We conclude that $\alpha\in {\rm Ker}({\rm res}_3^\prime)$. 
\end{proof}

\begin{proposition}\label{prop2.14}
Let $H=\langle a\rangle\simeq C_p$ $\lhd$ 
$H_0^\prime=\langle a,c\rangle\simeq (C_p)^2$ $\lhd$ 
$G=E_p(p^3)=\langle a,b,c\rangle\simeq (C_p)^2\rtimes C_p$. 
Then we have 
\begin{align*}
{\rm Ker}\{H^2(H_0^\prime,J_{G/H})\xrightarrow{\rm res_3^\prime}H^2(H,J_{G/H})\}\simeq (\bZ/p\bZ)^{\oplus p-2}. 
\end{align*}
\end{proposition}
\begin{proof}
For a normal subgroup $N\lhd G$ and $G$-module $M$, 
the Lyndon-Hochschild-Serre spectral sequence gives rise to 
the $7$-term 
exact sequence
\begin{align*}
0 &\to H^1(G/N,M^N)\xrightarrow{\rm inf} H^1(G,M)
\xrightarrow{\rm res} H^1(N,M)^{G/N} 
\xrightarrow{\rm tr} H^2(G/N,M^N) \\
&\xrightarrow{\rm inf} H^2(G,M)_1 
\xrightarrow{\rho} H^1(G/N,H^1(N,M))
\xrightarrow{\lambda} H^3(G/N,M^N)
\end{align*}
where $H^2(G,M)_1:={\rm Ker}\{H^2(G,M)
\xrightarrow{\rm res} H^2(N,M)\}$ 
(see Dekimpe, Hartl and Wauters \cite{DHW12}, see also Huebschmann \cite{Hue81a}, \cite{Hue81b}). 
We apply this to the case $G=H_0^\prime$, $N=H$ and $M=J_{G/H}$. 
Then we get 
\begin{align*}
H^1(H,J_{G/H})^{H_0^\prime/H}=0 
\xrightarrow{\rm tr} H^2(H_0^\prime/H,(J_{G/H})^{H})
\xrightarrow[\sim]{\rm inf} H^2(H_0^\prime,J_{G/H})_1 
\xrightarrow{\rho} H^1(H_0^\prime/H,H^1(H,J_{G/H}))=0
\end{align*}
where $H^2(H_0^\prime,J_{G/H})_1={\rm Ker}\{H^2(H_0^\prime,J_{G/H})\xrightarrow{\rm res_3^\prime}H^2(H,J_{G/H})\}$ 
because $H^1(H,J_{G/H})=0$ by Lemma \ref{lem2.10}. 
We also obtain that 
$H^2(H_0^\prime/H,(J_{G/H})^H)\simeq \widehat{H}^0(H_0^\prime/H,(J_{G/H})^H)\simeq (\bZ/p\bZ)^{\oplus p-2}$ 
because $(J_{G/H})^H$ is a permutation $H_0^\prime/H$-lattice with $p-2$ fixed points 
(see the proof of Lemma \ref{lem2.12} (4), (6)).  
Hence we have $H^2(H_0^\prime,J_{G/H})_1\simeq (\bZ/p\bZ)^{\oplus p-2}$.  
\end{proof}

\begin{proposition}\label{prop2.15}
Let $H=\langle a\rangle\simeq C_p$ $\lhd$ 
$H_0^\prime=\langle a,c\rangle\simeq (C_p)^2$ $\lhd$ 
$G=E_p(p^3)=\langle a,b,c\rangle\simeq (C_p)^2\rtimes C_p$. 
Then we have 
\begin{align*}
{\rm Ker}\{H^2(G,J_{G/H})\xrightarrow{{\rm res_3}^\prime\circ{\rm res_3}} H^2(H,J_{G/H})\}\simeq 
\langle \overline{f_2}\rangle\simeq\bZ/p\bZ. 
\end{align*}
\end{proposition}
\begin{proof}
It follows from Proposition \ref{prop2.11} via the $7$-term exact sequence 
that ${\rm res}_3: H^2(G,J_{G/H})\to H^2(H_0^\prime,J_{G/H})$ is injective. 
It also follows from Proposition \ref{prop2.14} via the $7$-term exact sequence again 
that ${\rm Ker}({\rm res}_3^\prime: H^2(H_0^\prime,J_{G/H})\to H^2(H,J_{G/H}))\simeq (\bZ/p\bZ)^{\oplus p-2}$. 
From the surjectivity of $\delta^\prime: H^2(H_0^\prime,J_{G/H})\simeq (\bZ/p\bZ)^{\oplus p-1}
\to H^3(H_0^\prime,\bZ)\simeq\bZ/p\bZ$, we have 
${\rm Ker}(\delta^\prime)\simeq (\bZ/p\bZ)^{\oplus p-2}$. 
Because $\delta$ is an isomorphism, it follows from the commutativity 
${\rm res}_4\circ\delta=\delta^\prime\circ{\rm res}_3$ and 
Lemma \ref{lem2.5} that 
$\delta({\rm Ker}({\rm res}_3^\prime\circ{\rm res}_3))=
\delta({\rm Ker}(\delta^\prime\circ{\rm res}_3))=
\delta({\rm Ker}({\rm res}_4\circ\delta))=
{\rm Ker}({\rm res}_4)=\langle \overline{f_2}\rangle$. 
\end{proof}

We complete the proof of Theorem \ref{thmain} 
via Proposition \ref{prop2.11} and Proposition \ref{prop2.15}. 
\end{proof}

\begin{example}[$p=3$]\label{ex2.16}
Let $G=E_p(p^3)=\langle a,b,c\rangle\simeq (C_p)^2\rtimes C_p$. 
Assume that $p=3$. 
For $H=\{1\}$, we get
\begin{align*}
&{\rm Ker}\{H^2(G,J_G)\xrightarrow{\rm res} H^2(\{1\},J_G)\}=H^2(G,J_G)=\langle f_1,f_2\rangle\simeq (\bZ/3\bZ)^{\oplus 2},\\
&{\rm Ker}\{H^2(G,J_G)\xrightarrow{\rm res} H^2(G,J_G)\}=0
\end{align*}
and\\

\renewcommand{\arraystretch}{1.2}
\begin{tabular}{c||c|cccc}
$H^\prime\simeq C_p$ 
& $Z(G)=\langle c\rangle$ & $H_0=\langle a\rangle$ & $H_1=\langle ab\rangle$ 
& $H_2=\langle ab^2\rangle$ & $H_3=\langle b\rangle$\\\hline
${\rm Ker}\{H^2(G,J_G)\xrightarrow{\rm res} H^2(H^\prime,J_G)\}$ 
& $(\bZ/3\bZ)^{\oplus 2}$ & $(\bZ/3\bZ)^{\oplus 2}$ & $(\bZ/3\bZ)^{\oplus 2}$ & $(\bZ/3\bZ)^{\oplus 2}$ & $(\bZ/3\bZ)^{\oplus 2}$
\end{tabular}\vspace*{2mm}

\begin{tabular}{c||cccc}
$H^\prime\simeq (C_p)^2$
& $H_0^\prime=\langle a,c\rangle$ & $H_1^\prime=\langle ab,c\rangle$ & 
$H_2^\prime=\langle ab^2,c\rangle$ & $H_3^\prime=\langle b,c\rangle$\\\hline
${\rm Ker}\{H^2(G,J_G)\xrightarrow{\rm res} H^2(H^\prime,J_G)\}$ 
& $\langle f_2\rangle\simeq \bZ/3\bZ$ & $\langle f_1f_2^2\rangle\simeq \bZ/3\bZ$ 
& $\langle f_1f_2\rangle\simeq \bZ/3\bZ$ & $\langle f_1\rangle\simeq \bZ/3\bZ$ 
\end{tabular}\vspace*{2mm}~\\

For $H=\langle a\rangle\simeq C_p$, we get
\begin{align*}
&{\rm Ker}\{H^2(G,J_{G/H})\xrightarrow{\rm res} H^2(\{1\},J_{G/H})\}=H^2(G,J_{G/H})=\langle f_1,f_2\rangle\simeq (\bZ/3\bZ)^{\oplus 2},\\
&{\rm Ker}\{H^2(G,J_{G/H})\xrightarrow{\rm res} H^2(G,J_{G/H})\}=0
\end{align*}
and\\

\begin{tabular}{c||c|cccc}
$H^\prime\simeq C_p$
& $Z(G)=\langle c\rangle$ & $H_0=\langle a\rangle$ & $H_1=\langle ab\rangle$ 
& $H_2=\langle ab^2\rangle$ & $H_3=\langle b\rangle$\\\hline
${\rm Ker}\{H^2(G,J_{G/H})\xrightarrow{\rm res} H^2(H^\prime,J_{G/H})\}$ & 
$(\bZ/3\bZ)^{\oplus 2}$ & $\langle f_2\rangle\simeq\bZ/3\bZ$ & $(\bZ/3\bZ)^{\oplus 2}$ & $(\bZ/3\bZ)^{\oplus 2}$ & $(\bZ/3\bZ)^{\oplus 2}$
\end{tabular}\vspace*{2mm}

\begin{tabular}{c||cccc}
$H^\prime\simeq (C_p)^2$
& $H_0^\prime=\langle a,c\rangle$ & $H_1^\prime=\langle ab,c\rangle$ & 
$H_2^\prime=\langle ab^2,c\rangle$ & $H_3^\prime=\langle b,c\rangle$\\\hline
${\rm Ker}\{H^2(G,J_{G/H})\xrightarrow{\rm res} H^2(H^\prime,J_{G/H})\}$ 
& $0$ & $\langle f_1f_2^2\rangle\simeq \bZ/3\bZ$ 
& $\langle f_1f_2\rangle\simeq \bZ/3\bZ$ & $\langle f_1\rangle\simeq \bZ/3\bZ$
\end{tabular}\vspace*{2mm}
\renewcommand{\arraystretch}{1}
\end{example}

As a consequence of Theorem \ref{thmain}, 
we get the 
Tamagawa number $\tau(T)$ of 
$T=R^{(1)}_{K/k}(\bG_{m,K})$ of $K/k$ over a global field $k$ 
via Ono's formula $\tau(T)=|H^1(k,\widehat{T})|/|\Sha(T)|
=|H^1(G,J_{G/H})|/|\Sha(T)|$ 
(see Ono \cite[Main theorem, page 68]{Ono63}, \cite{Ono65}, 
Voskresenskii \cite[Theorem 2, page 146]{Vos98} and 
Hoshi, Kanai and Yamasaki \cite[Section 8, Application 2]{HKY22}). 
\begin{corollary}\label{cor2.17}
Let the notation be as in Theorem \ref{thmain}. 
Then the Tamagawa number 
\begin{align*}
\tau(T)=
\begin{cases}
p^2/|\Sha(T)| & {\rm if}\quad H=\{1\},\\
p/|\Sha(T)| & {\rm if}\quad H\simeq C_p\ 
{\rm with}\ H\neq Z(G)=\langle c\rangle
\end{cases}
\end{align*}
where $\Sha(T)\leq (\bZ/p\bZ)^{\oplus 2}$ $($resp. $\Sha(T)\leq \bZ/p\bZ$$)$ 
when $H=\{1\}$ $($resp. $H\simeq C_p$ with $H\neq Z(G)$$)$ 
is given as in Theorem \ref{thmain}. 
\end{corollary}
\begin{proof}
As in the proof of Lemma \ref{lem2.12} (1), by Proposition \ref{prop2.2}, 
we get $H^1(G,\bZ[G/H])=0\to H^1(G,J_{G/H})$ 
$\to$ $H^2(G,\bZ)\simeq G^{ab}\simeq (\bZ/p\bZ)^{\oplus 2}\to 
H^2(G,\bZ[G/H])\simeq H^{ab}\xrightarrow{0} 
H^2(G,J_{G/H})\xrightarrow[\sim]{\delta}H^3(G,\bZ)$. 
If $H=\{1\}$, then we have 
$H^1(G,J_G)\simeq G^{ab}\simeq (\bZ/p\bZ)^{\oplus 2}$ because 
$H^{ab}=\{1\}$. 
If $H\simeq C_p$ with $H\neq Z(G)$, 
then $H^1(G,J_{G/H})\simeq \bZ/p\bZ$ because 
$H^{ab}\simeq \bZ/p\bZ$. 
Applying Ono's formula $\tau(T)=|H^1(k,\widehat{T})|/|\Sha(T)|=|H^1(G,J_{G/H})|/|\Sha(T)|$ (see Ono \cite[Main theorem, page 68]{Ono63}), 
we get the assertion. 
\end{proof}
%


\end{document}